\documentclass[11pt]{amsart}
\usepackage{amssymb}
\usepackage{multirow}
\usepackage{mathtools}
\usepackage{hyperref}
\hypersetup{breaklinks=true}
\usepackage{breakurl}
\usepackage{enumitem}
\usepackage{tikz-cd}
\usepackage{mathrsfs}

\newcommand{\C}{\mathbb{C}}

\newcommand{\Q}{\mathbb{Q}}

\newcommand{\Z}{\mathbb{Z}}

\newcommand{\dec}{\mathrm{dec}}
\newcommand{\Adec}{\mathcal{A}_g^{\mathrm{dec}}[n]}
\newcommand{\Adeclam}{\mathcal{A}_{g,\lambda}^{\mathrm{dec}}[n]}
\newcommand{\Amdec}{\mathcal{A}_g^{\mathrm{mdec}}[n]}

\DeclareMathOperator{\End}{End}

\DeclareMathOperator{\Hom}{Hom}

\usepackage[margin=1in]{geometry}

\newtheorem{thm}{Theorem}[section]
\newtheorem{cor}[thm]{Corollary}
\newtheorem{prop}[thm]{Proposition}
\newtheorem{lem}[thm]{Lemma}

\theoremstyle{definition}
\newtheorem{defn}[thm]{Definition}

\newtheorem{example}[thm]{Example}

\theoremstyle{remark}
\newtheorem{rem}[thm]{Remark}

 \newtheoremstyle{TheoremNum}
        {7pt}{7pt}              
        {\itshape}                      
        {}                              
        {\bfseries}                     
        {.}                             
        { }                             
        {\thmname{#1}\thmnote{ \bfseries #3}}
    \theoremstyle{TheoremNum}

     \newtheoremstyle{TheoremNum}
        {7pt}{7pt}              
        {\itshape}                      
        {}                              
        {\bfseries}                     
        {.}                             
        { }                             
        {\thmname{#1}\thmnote{ \bfseries #3}}
    \theoremstyle{TheoremNum}
    \newtheorem{corn}{Corollary}

\begin{document}

\title{Complete Families of Indecomposable Non-simple Abelian Varieties}
\author{Laure Flapan}
\address{Department of Mathematics, Michigan State University, East Lansing, MI 48824}
\email{flapanla@msu.edu}

\subjclass[2010]{14D05, 14D07, 14C30, 11G15, 14H10}
\keywords{monodromy group, variation of Hodge structures, abelian variety, Kodaira fibration }

\begin{abstract}
Given a fixed product of non-isogenous abelian varieties at least one of which is general, we show how to construct complete families of indecomposable abelian varieties whose very general fiber is isogenous to the given product and whose connected monodromy group is a product of symplectic groups or is a unitary group. As a consequence, we show how to realize any product of symplectic groups of total rank $g$ as the connected monodromy group of a complete family of $g'$-dimensional abelian varieties for any $g'\ge g$. These methods also yield a construction of a new Kodaira fibration with fiber genus $4$. 

\end{abstract}
\maketitle

\maketitle


\section{Introduction}\label{intro}
If $C$ is a smooth algebraic curve, then its Jacobian $J(C)$ is indecomposable as a principally polarized abelian variety, meaning that $J(C)$ cannot be written as the product of smaller-dimensional principally polarized abelian varieties. However, if $\tilde{C}$ is another smooth algebraic curve equipped with a finite map $\tilde{C}\rightarrow C$, then although $J(\tilde{C})$ must also be indecomposable, the Jacobian $J(\tilde{C})$ is not simple as an abelian variety because $J(\tilde{C})$ is isogenous to a product of smaller-dimensional principally polarized abelian varieties, one of which is $J(C)$. 

In this paper, we investigate the phenomenon of indecomposable yet non-simple abelian varieties in families in the context of studying monodromy. For any family $f\colon \mathscr{A}\rightarrow B$ of complex abelian varieties, where $B$ is an algebraic variety and $A_b$ is a very general fiber, the \emph{connected monodromy group} $T(\mathscr{A})$ is the connected component of the identity of the $\Q$-Zariski closure of the image of the monodromy representation $\Phi\colon \pi_1(B)\rightarrow GL(H^1(A_b,\Q))$. Our interest is thus in studying complete families of indecomposable non-simple abelian varieties with specified monodromy. We prove:

\begin{thm}\label{main theorem}
Consider any fixed product of pairwise non-isogenous abelian varieties $A_{c,1}\times \cdots \times A_{c,r}\times A_{v,1}\times \cdots \times A_{v,s}$ such that each $A_{v,j}$ is general in the moduli space $\mathcal{A}_{g_{v,j}}$ and all $g_{v,j}\ge 2$. Then for any positive integer $d\le \min_j g_{v,j}-1$, there exist infinitely many complete families  $f\colon \mathscr{A}\rightarrow B$ of maximal variation of $g$-dimensional principally polarized abelian varieties with the properties:
\begin{enumerate}
\item The base $B$ is a smooth projective variety of dimension $d$
\item Every fiber of $f$ is indecomposable.
\item If $A$ is a very general fiber of $f$, then $A$ is isogenous to the product 
\[A_{c,1}\times \cdots \times A_{c,r}\times A_{v,1}\times \cdots \times A_{v,s}.\]
\item The connected monodromy group of $f\colon \mathscr{A}\rightarrow B$ is 
\[T(\mathscr{A})=\prod_{j=1}^{s}Sp_\Q(V_{v,j}),\]
where $V_{v,j}=H^1(A_{v,j},\Q)$ for each $1\le j\le s$.
\end{enumerate}
\end{thm}

The families $f\colon \mathscr{A}\rightarrow B$ being of  ``maximal variation'' means simply that the resulting moduli map $B\rightarrow \mathcal{A}_g$ is generically finite. For the general definition of the variation of a family see \cite{viehweg}.

One motivation for this result comes from the study of the possible fundamental groups of a smooth projective variety, or relatedly, the possible monodromy representations of a family of smooth projective varieties.  Simpson shows in \cite{simpson} that any representation of the fundamental group of a smooth projective variety into a reductive group can be deformed to one coming from a variation of Hodge structures. Recently, Arapura has shown  \cite[Section 2]{arapura} that if $f\colon X\rightarrow Y$ is a smooth projective morphism of smooth projective varieties with fiber $F$ and $\rho$ is the corresponding monodromy representation, then for every choice of normal subgroup $N$ of $\pi_1(F)$ and character $\chi$ of the quotient, the representation $\rho$ gives rise to a representation $\rho^{N, \chi}$ which is the monodromy representation of a polarizable variation of $\Q$-Hodge structures of type $\{(-1,0), (0,-1)\}$ over some $\tilde{Y}$ such that there is a surjective generically finite morphism of smooth varieties $p\colon \tilde{Y}\rightarrow Y$.  

Using  the equivalence between polarizable $\Q$-Hodge structures of type $\{(-1,0), (0,-1)\}$ and complex abelian varieties up to isogeny, it is thus natural to consider monodromy representations arising from families of abelian varieties and ask what groups can arise as their connected monodromy groups. With this motivation in mind, Theorem \ref{main theorem} yields the following consequence:

\begin{cor}\label{main group cor}
For any positive integer $d\le \min_i g_i-1$, where $g_i\ge 2$ for all $i$, and any $g'\ge \sum_{i=1}^sg_i$, the group  
\[T=\prod_{j=1}^sSp(2g_i,\Q)\]
arises as the connected monodromy group of a $d$-dimensional family of maximal variation of  indecomposable $g'$-dimensional abelian varieties over a smooth projective base.
\end{cor}

\subsection{Kodaira fibrations}\label{intro kod section}
The motivation for focusing on families whose fibers are indecomposable non-simple abelian varieties in all of the above results comes from questions about the monodromy of Kodaira fibrations. Note that by repeated use of the Lefschetz hyperplane theorem,  any group which is the fundamental group of a smooth projective variety is in fact the fundamental group of an algebraic surface (and in fact of a surface of general type). 

Kodaira fibrations are a particular class of algebraic surface of general type given by a non-isotrivial fibration $f\colon S\rightarrow B$ to a smooth algebraic curve $B$ such that all fibers $F$ of $f$ are smooth algebraic curves. The fundamental group of such a surface $S$ is then determined by the induced exact sequence
$1\rightarrow \pi_1(F)\rightarrow \pi_1(S)\rightarrow \pi_1(B)\rightarrow 1.$
This extension is in turn determined by the induced homomorphism $\Phi\colon \pi_1(B)\rightarrow \mathrm{Mod}(F)\subset O(\pi_1(F))$
into the mapping class group of $F$ (using that the genus $g$ of $F$ is at least $3$ \cite{kas} and so the center $Z(\pi_1(F))$ is trivial \cite{FM} together with the fact that extensions with outer action  $\Phi$ are parametrized by $H^2(\pi_1(B),  Z(\pi_1(F_b)))$ \cite[Corollary 6.8]{brown}). So it is natural to try to understand the possible connected monodromy groups of the resulting $\overline{\Phi}\colon \pi_1(B)\rightarrow GL(H^1(F,\Q))$ in the context of the broader question of understanding the fundamental groups of smooth projective varieties. 

Unfortunately there are few known constructions of Kodaira fibrations and in general their monodromy representations are poorly understood. The first examples of such surfaces discovered by Kodaira \cite{kodsur}, Atiyah \cite{atiyah}, and Hirzebruch \cite{hirzebruch} as well as subsequent constructions by Kas \cite{kas}, Riera \cite{riera}, Gonzalez-Diez--Harvey \cite{gonzalez}, Zaal \cite{zaal}, Bryan--Donagi \cite{BD}, Cattanese--Rollenske \cite{CR}, and Lee--L\"onne--Rollenske \cite{LLR} all rely on taking a suitable ramified cover of a product $C_1\times C_2$ of curves. In such constructions, the fiber $F$ is  a ramified cover of one of the two curves in the product, say $C_1$. Hence although the Jacobian $J(F)$ is indecomposable since $F$ is smooth, we know that $J(F)$ is non-simple in the sense that $J(F)$ is isogenous to a product of abelian varieties, one of which is $J(C_1)$. Thus understanding the behavior of monodromy representations of complete families of indecomposable non-simple abelian varieties over a curve has the potential to shed some light on the question of monodromy of Kodaira fibrations. 

In fact, because the Torelli locus in genus $g=3$ and $g=4$ is of low codimension in $\mathcal{A}_g$, the construction techniques developed in the proof of Theorem \ref{main theorem} lead to a reproof (Corollary \ref{cor kod reproof}) of a construction given in \cite{FKod} of a Kodaira fibration of fiber genus $3$ with connected monodromy group $Sp(4)$ as well as a new construction of a Kodaira fibration with fiber genus $4$:

\begin{corn}[\ref{kod cor}] 
For any choice of elliptic curve $E$, there exist infinitely many Kodaira fibrations $f\colon S\rightarrow B$ with fiber genus $4$ such that a very general fiber $F$ decomposes as $H^1(F,\Q)=H^1(E,\Q)\oplus W$ and the connected monodromy group of $f\colon S\rightarrow B$ is $Sp_\Q(W)\cong Sp(6,\Q)$. 
\end{corn}

This provides a new addition to the limited collection of Kodaira fibrations where something about the monodromy representation is understood. 

\subsection{Ideas from the proof of Theorem \ref{main theorem}}
The main tool used to prove Theorem \ref{main theorem}  is the language of Shimura varieties and Hecke correspondences, which we introduce in Sections \ref{section 1} and \ref{section 2}. Essentially, any element  $\alpha \in Sp(2g, \Q)$ defines a correspondence, called a \emph{Hecke correspondence} and also denoted $\alpha$, on the moduli space $\mathcal{A}_g$. The main idea is to show that if $\alpha$ is sufficiently general in $Sp(2g,\Q)$, then the image under this $\alpha$ of the Shimura variety $\mathcal{Z}$ parametrizing principally polarized abelian varieties with the desired decomposition intersects the decomposable locus $\mathcal{A}_g^{\mathrm{dec}}$ in high codimension (see Corollary \ref{main cor}). 

In fact, in order to work with a fine moduli space we will need to impose a level structure and consider the analogous Shimura variety $\mathcal{Z}$ sitting inside $\mathcal{A}_g[n]$ for some $n\ge 3$, where $\mathcal{A}_g[n]$ denotes the moduli space of principally polarized abelian varieties with fixed level $n$ structure. More precisely, Corollary \ref{main cor} shows that if $\alpha$ is general, then $\alpha \mathcal{Z}\cap \Adec\subset \Amdec$, where $\Amdec$ is the locus of principally polarized abelian varieties with two isogenous simple factors. 

We produce the families of abelian varieties with the desired properties by taking complete intersections of ample divisors on the Satake--Baily--Borel compactification $\mathcal{Z}^*$ of $\mathcal{Z}$. Since both the boundary $\mathcal{Z}^*\backslash \mathcal{Z}$ of $\mathcal{Z}^*$ and the intersection $\mathcal{Z}\cap \Amdec$ occur in high codimension, such a complete intersection $X$ will have the property that $\alpha X\cap \Adec=\emptyset$ (see Proposition \ref{bb prop}).

Computing the connected monodromy group of the corresponding family $\mathscr{A}$ of abelian varieties  follows from a result of Andr\'e \cite[Theorem 1]{andre} which says that the connected monodromy group of a variation of Hodge structures over a connected algebraic variety is a normal subgroup of the derived subgroup of the Mumford-Tate group of a very general fiber. Since a very general fiber of $\mathscr{A}$ is by construction isogenous to a very general point of $\mathcal{Y}$ (and hence they have isomorphic $\Q$-Hodge structures), one can use the properties of the Shimura variety $\mathcal{Z}$ to compute the derived Mumford-Tate group of this fiber. In the cases we consider, we then show that the connected monodromy group is in fact equal to this derived Mumford-Tate group.

\subsection{Other connected monodromy groups}
The strategy of the proof of Theorem \ref{main theorem} can also be used to produce complete families of indecomposable abelian varieties whose connected monodromy groups are other algebraic groups, not only products of symplectic groups. The particular combinatorics involved to compute the codimension of the boundary of Satake--Baily--Borel compactification of $\mathcal{Z}$ as well as the codimension of the locus $\alpha \mathcal{Z}\cap \Adec$ vary depending on the desired properties of the general fiber of the family. 

As an illustration of how the techniques of Theorem \ref{main theorem} can be used to produce complete families of indecomposable abelian varieties with other kinds of connected monodromy groups, we prove 
 a result similar to  Theorem \ref{main theorem} in Theorem \ref{thm unitary} in the case of the connected monodromy group being a unitary group. Rather than prove a full analogue of Theorem \ref{main theorem}, we simplify the theorem statement and proof by only considering the case when the fixed part  $A_{c,1}\times \cdots \times A_{c,r}$ of the family consist of pairwise non-isogenous elliptic curves without CM and the varying part consists of a single simple abelian variety with endomorphisms by a fixed  imaginary quadratic field $L$.  

In fact, a much more general version of Theorem \ref{thm unitary} which would be more directly analogous to Theorem \ref{main theorem} should be true, however the dimension-counting combinatorics become more involved and the issues involving the computation of the Mumford-Tate group become more subtle when multiple imaginary quadratic fields are involved.

The precise language of abelian varieties with endomorphisms by an imaginary quadratic field will be introduced in Sections \ref{section 1} and \ref{section unitary}, where we will introduce the moduli space $\mathcal{Z}_L(p,q)[n]$ of $p+q$-dimensional principally polarized abelian varieties with level structure $n$ that have endomorphisms by the imaginary quadratic field $L$ via a fixed action determined by positive integers $(p,q)$. 

Rather than state here the precise $\Q$-group arising as the connected monodromy group of the families we produce, we can summarize Theorem \ref{thm unitary} as:

\begin{thm}
Fix pairwise non-isogenous elliptic curves $E_1, \ldots, E_r$ without CM as well as an imaginary quadratic field $L$, a choice of positive integers $p,q$ such that $p+q\ge4$, and an abelian variety $Z$ which is general in the moduli space $\mathcal{Z}_L(p,q)[n]$. 
Then for any positive integer
\[d\le \min(2p, p+q-2, 2q)-1\]
 there exist infinitely many complete families  $f\colon \mathscr{A}\rightarrow S$ of maximal variation of abelian varieties of dimension $g=p+q+r$ with the properties:
\begin{enumerate}
\item The base $S$ is a smooth projective variety of dimension $d$
\item Every fiber of $f$ is indecomposable.
\item If $A$ is a very general fiber of $f$, then $A$ is isogenous to the product
$E_1\times \cdots \times E_r\times Z,$ where $\End_\Q(Z)=L$ with prescribed action determined by $(p,q)$
\item Letting $V=H^1(Z,\Q)$, the connected monodromy group of $\mathscr{A}$ is a $\Q$-form of $SU(p,q)$.
\end{enumerate}
\end{thm}

In analogy with Corollary \ref{main group cor} above, we deduce Corollary \ref{unitary cor}, which we summarized here as:

\begin{cor}
For any positive integers $p,q,$ and $g'$ such that $5\le p+q+1\le g'$ and any positive integer $d\le \min(2p, p+q-2, 2q)-1$, there is a $\Q$ form of the group $SU(p,q)$ which arises as the connected monodromy group of a $d$-dimensional complete family of maximal variation of $g'$-dimensional indecomposable abelian varieties over a smooth algebraic variety.
\end{cor}

\subsection{Organization}

The organization of the paper is as follows. In Section \ref{section 1} we introduce the Satake--Baily--Borel compactification of a Shimura variety and the principal examples of Shimura varieties and their Satake--Baily--Borel compactifications that will be of use to us. Section \ref{section 2} then considers the decomposable locus $\Adec$ as well as the locus $\Amdec$ of abelian varieties with two isogenous simple factors and calculates the dimension of intersection of these loci with the Shimura varieties of interest.  

Then in Section \ref{Hecke section}, we introduce the language of Hecke correspondences and their relationship with isogenies of abelian varieties. In particular, we prove in Lemma \ref{polarized isogeny lemma} that an abelian variety $A$ whose endomorphism algebra is a product of fields which are either $\Q$ or an imaginary quadratic field has the property that if $A$ is isogenous  to another abelian variety $B$, then there is a Hecke correspondence $\alpha$ such that $\alpha A=B$. Section \ref{sec: Gamma} then shows that we may determine which elements of $Sp(2g,\Q)$ send points of $\Adec \backslash \Amdec$ to $\Adec$ by determining which such elements send a product of elliptic curves to $\Adec$. We then prove in Lemma \ref{dim lemma} that the elements of $Sp(2g,\Q)$ with this property form a proper subvariety and thus (Corollary \ref{main cor}) a general element $\alpha$ of $Sp(2g,\Q)$ has the property  
$\alpha \mathcal{Y}\cap \Adec\subset \Amdec$.

Section \ref{section MT}  introduces Mumford-Tate groups and Proposition \ref{bb prop} determines conditions under which one may produce a complete family of abelian varieties whose connected monodromy group is determined by the Mumford-Tate group of the general element of the family. The proofs of Theorem \ref{main theorem}  and Corollary \ref{main group cor} then occur in Section \ref{section main}. Section \ref{section kodaira} discusses the applications to Kodaira fibrations and in particular the new construction (Corollary \ref{kod cor}) of a Kodaira fibration of fiber genus $4$. 

Lastly Section \ref{section unitary} discusses constructions of families with other connected monodromy groups, in particular focusing on the example of a unitary connected monodromy group in the case that the varying part of the family consists of a single abelian variety with endomorphisms by an imaginary quadratic field.

\section{Shimura Varieties and the Satake--Baily--Borel compactification}\label{section 1}
Recall that a (connected) Shimura variety $Z$ is the quotient of a Hermitian symmetric domain $X^+$ by a congruence subgroup $\Gamma$ of some reductive algebraic group $G$ defined over $\mathbb{Q}$ such that $G(\mathbb{R})$ acts on $X^+$ by conjugation. 

The Satake--Baily--Borel compactification $(\Gamma\backslash X^+)^*$ of $\Gamma\backslash X^+$, first constructed in \cite{BB}, has the structure of a projective algebraic variety and the property that the boundary consists of finitely many explicitly described  Shimura varieties attached to particular subgroups of $G$.

The principal examples of Shimura varieties that will be of interest to us in this paper are the following:

\begin{example}\label{QBB}
The moduli space $\mathcal{A}_g=Sp(2g,\Z)\backslash \mathbb{H}_g$ of principally polarized abelian varieties $(A,\lambda)$ of dimension $g$, where $\mathbb{H}_g$ denotes the Siegel upper half-space of degree $g$, is a connected Shimura variety of dimension $\frac{1}{2}g(g+1)$. Since this $\mathcal{A}_g$ is only a coarse moduli space, in order to work with a fine moduli space we will need to impose a level structure and consider, for some fixed $n\ge 3$ the Shimura variety
\[\mathcal{A}_g[n]=\ker(Sp(2g,\Z)\rightarrow Sp(2g, \Z/n\Z))\backslash \mathbb{H}_g.\]
This Shimura variety $\mathcal{A}_g[n]$ is the moduli space of principally polarized abelian varieties with this fixed level structure $n\ge 3$. 
The Satake--Baily--Borel compactification $\mathcal{A}_g[n]^*$ of $\mathcal{A}_g[n]$ has the property that its boundary is the union of strata
\[\bigcup_{1\le i\le g-1} \mathcal{A}_i[n].\]
In particular, the boundary $\mathcal{A}_g[n]^*\backslash \mathcal{A}_g[n]$ has codimension in  $\mathcal{A}_g[n]^*$
\[\frac{1}{2}g(g+1)-\frac{1}{2}(g-1)(g)=g.\]

\end{example}

\begin{example}\label{unitaryBB}
Let $L$ be an imaginary quadratic field such that $L\hookrightarrow \End_\Q(A)$ for some simple abelian variety $A$ of dimension $g$. Let $\sigma, \overline{\sigma} \colon L\hookrightarrow \C$ denote the two embeddings of $L$ into $\C$. Let $V=H^1(A,\Q)$ and consider the action of $L\otimes_\Q \C$ on $V_\C\coloneqq V\otimes_\Q \C$. 

Then $V_\C$ decomposes as a sum $V_\C(\sigma)\oplus V_\C(\overline{\sigma})$, where $V_\C(\sigma)=\{v\in V_\C\mid \ell\cdot v=\sigma(\ell)v \  \forall \ \ell \in L\}$ and similarly for $V_\C(\overline{\sigma})$. Letting $p=\dim V_\C(\sigma)$ and $q=\dim V_\C(\overline{\sigma})$, we then have $p+q=g$. 

For such fixed positive integers $p, q$ such that $p+q=g$, consider the Hermitian symmetric domain
\[\mathcal{D}(p,q)=\{z\in M_{p,q}(\mathbb{C})\mid 1-z^t\overline{z} \text{ is positive Hermitian}\}.\]
Then for a given imaginary quadratic field $L$ the Shimura variety
\[\mathcal{Z}_L(p,q)[n]=\ker(SU((p,q),\mathcal{O}_L)\rightarrow SU((p,q), \mathcal{O}_L/n\mathcal{O}_L))\backslash \mathcal{D}(p,q)\]
parametrizes $g$-dimensional principally polarized abelian varieties $A$ with level $n$ structure such that $L\hookrightarrow \End_\Q(A)$ with given action described by the integers $(p,q)$. This Shimura variety $\mathcal{Z}_L(p,q)[n]$ has dimension $pq$. 

The strata of the boundary of the Satake--Baily--Borel compactification $\mathcal{Z}_L(p,q)[n]^*$ are of the form $\mathcal{Z}_L(p-r,q-r)[n]$ for $r>0$. Hence  the codimension of  the boundary $\mathcal{Z}_L(p,q)[n]^*\backslash \mathcal{Z}_L(p,q)[n]$ in $\mathcal{Z}_L(p,q)[n]^* $ is at least
\[pq-(p-1)(q-1)=(p+q)-1=g-1.\]

\end{example}


\section{The Decomposable Locus of Abelian Varieties}\label{section 2}
Fix an $n\ge 3$ and let $\mathcal{A}_g[n]$ be the fine moduli space of principally polarized abelian varieties $(A,\lambda)$ with level structure $n$. Let $\mathcal{A}_g^{\dec}[n]$ denote the decomposable locus of $ \mathcal{A}_g[n]$, meaning that
\[\mathcal{A}_g^{\dec}[n]=\{(A, \lambda)\in \mathcal{A}_g[n]\mid (A,\lambda)=(A_1,\lambda_1)\times (A_2,\lambda_2) \text{ st  } A_i\in A_{g_i}[n], \ g_i<g \text{ for } i=1,2\}.\]
In the above as well as throughout the paper we use the notation $\sim$ to mean ``isogenous to''.

We will also want to consider the locus of non-simple abelian varieties that have two isogenous factors. We denote this locus 
\[\Amdec=\{A\in \mathcal{A}_g[n]\mid A\sim \prod_{i=1}^k A_i^{n_i} \text{ and } n_i\ge 2 \text{ for some } 1\le i\le k\}.\]

The purpose of this section is to collect a number of combinatorial computations describing the dimension of intersection of $\Amdec$ with the sub-loci of $\mathcal{A}_g^{\dec}[n]$ that will be of interest to us in what follows. 

\subsection{ The locus $\Amdec$ and products  $\mathcal{A}_{g_1}[n]\times \cdots \times \mathcal{A}_{g_r}[n]$}
\begin{lem}\label{lem dim1}
Consider a product of moduli spaces in  $\mathcal{A}_g[n]$
\[\mathcal{Z}=\mathcal{A}_{g_1}[n]\times \cdots \times \mathcal{A}_{g_r}[n],\]
 where $2\le g_1\le \cdots \le g_r$ and   $\sum_{i=1}^rg_i=g$. Then $\mathcal{Z}\cap \Amdec$ has codimension $2g_1-2$ in $\mathcal{Z}$.
\end{lem}

\begin{proof}
We will write points of $\mathcal{Z}$ as products $A_1\times \cdots \times A_r$ of abelian varieties where by assumption $A_k\in \mathcal{A}_{g_k}[n]$ for all $1\le k\le r$.
For $i\ne j$ define
\[ \mathcal{B}_{i,j,d}=\{A_1\times \cdots \times A_r\in \mathcal{Z}\mid \exists Z\in \mathcal{A}_d[n] \text{ st } A_i\sim Z\times A_i', \ A_j\sim Z\times A_j' \text{ for some } A_i'\in \mathcal{A}_{g_i-d}[n], A_j'\in\mathcal{A}_{g_j-d}[n] \}.\]
Similarly define
\[ \mathcal{B}_{i,i,d}=\{A_1\times \cdots \times A_r\in \mathcal{Z}\mid \exists Z\in \mathcal{A}_d[n] \text{ st } A_i\sim Z^2\times A_i' \text{ for some } A_i'\in \mathcal{A}_{g_i-2d}[n]\}.\]

Then observe that 
\begin{equation}\label{eq: Zunion}\mathcal{Z}\cap \Amdec=\bigcup_{i,j,d} \mathcal{B}_{i,j,d}.\end{equation}

Note that if $i\ne j$ we have
\begin{equation*}
\begin{aligned}
\dim \mathcal{B}_{i,j,d}&=\dim \mathcal{Z}-\dim \mathcal{A}_{g_i}[n]-\dim \mathcal{A}_{g_j}[n] + \dim \mathcal{A}_{g_i-d}[n]+ \dim \mathcal{A}_d[n]+ \dim \mathcal{A}_{g_j-d}[n]\\
&=\dim \mathcal{Z} - \frac{g_i(g_i+1)}{2}-\frac{g_j(g_j+1)}{2}+ \frac{(g_i-d)(g_i-d+1)}{2}+\frac{d(d+1)}{2} + \frac{(g_j-d)(g_j-d+1)}{2}\\
&=\dim \mathcal{Z}-\frac{1}{2}d(2g_i+2g_j+1-3d)
\end{aligned}
\end{equation*}
Hence when $i\ne j$, since $d\le g_i, g_j$ and $g_i, g_j\ge 2$, the component $\mathcal{B}_{i,j,d}$ of $\mathcal{Z}$ is of minimal codimension when $d=1$.  In this case, the component $\mathcal{B}_{i,j,1}$ is of codimension $g_i+g_j-1$ in $\mathcal{Z}$. It follows that a component $\mathcal{B}_{i,j,d}$ of $\mathcal{Z}\cap \Amdec$ with minimal codimension in $\mathcal{Z}$ is the component $\mathcal{B}_{1,2,1}$, which has codimension $g_1+g_2-1$. 

In the case $i=j$,  we have
\begin{equation*}
\begin{aligned}
\dim \mathcal{B}_{i,i,d}&=\dim \mathcal{Z}-\dim \mathcal{A}_{g_i}[n]+ \dim \mathcal{A}_{g_i-2d}[n] + \dim \mathcal{A}_{d}[n]\\
&=\dim \mathcal{Z}-\frac{g_i(g_i+1)}{2}
+\frac{(g_i-2d)(g_i-2d+1)}{2}+\frac{d(d+1)}{2}\\
&=\dim \mathcal{Z}-\frac{1}{2}d(4g_i+1-5d)
\end{aligned}
\end{equation*}
Hence when $i=j$, since $2d\le g_i$ the component $\mathcal{B}_{i,j,d}$ of $\mathcal{Z}$ is of minimal codimension when $d=1$.  In this case $\mathcal{B}_{i,i,1}$ is of codimension $2g_i-2$. It follows that a component of the form $\mathcal{B}_{i,i,d}$ of $\mathcal{Z}\cap \Amdec$ with minimal codimension in $\mathcal{Z}$ is the component $\mathcal{B}_{1,1,1}$, which has codimension $2g_1-2$. Note that the assumption that $2\le g_1$ ensures that $\mathcal{B}_{1,1,1}\ne \emptyset$. 

Since $2g_1-2<g_1+g_2-1$, it follows from \eqref{eq: Zunion} that $\mathcal{B}_{1,1,1}$ is a component of $\mathcal{Z}\cap \Amdec$ of minimal codimension in $\mathcal{Z}$ and so $\mathcal{Z}\cap \Amdec$ has codimension $2g_1-2$ in $\mathcal{Z}$. 
\end{proof}

\begin{lem}\label{lem dim2}
Fix a product $P=A_{c,1}\times \cdots \times A_{c,r}$ of simple pairwise non-isogenous abelian varieties such that each $A_{c,i}\in \mathcal{A}_{g_{c,i}}[n]$ and $1\le g_{c,1}\le \cdots \le g_{c,r}$.
Consider a subvariety of $\mathcal{A}_g[n]$ of the form
\[\mathcal{Y}=\{P\}\times \mathcal{A}_{g_{v,1}}[n]\times \cdots \times \mathcal{A}_{g_{v,s}}[n],\]  
where $g_{v,1}\le \cdots \le g_{v,r}$. Then $\mathcal{Y}\cap \Amdec$ has codimension in $\mathcal{Y}$ given by
\[\min(2g_{v,1}-2,\frac{1}{2}g_{c,1}(2g_{v,1}+1-g_{c,1}), \frac{1}{2}g_{c,r}(2g_{v,1}+1-g_{c,r})).\] 
\end{lem}

\begin{proof}
We may consider $\mathcal{Y}$ as the product $\{P\}\times \mathcal{Z}$, where $\mathcal{Z}$ is the product of moduli spaces introduced in Lemma \ref{lem dim1}. 
Write points of $\mathcal{Y}$ as products $P\times A_1\times \cdots \times A_r$ of abelian varieties where by assumption $A_k\in \mathcal{A}_{g_{v,k}}[n]$ for all $1\le k\le s$. Then define
\[\mathcal{C}_{i}=\{P\times A_1\times \cdots \times A_r\in \mathcal{Y}\mid A_i\sim A_{c,j}\times A_i' \text{ for some } 1\le j\le r \text{ and } A_i'\in \mathcal{A}_{g_{v,i}-g_{c,j}}[n]\}.\]

Then observe that
\begin{equation}\label{eq: Yunion}\mathcal{Y}\cap \Amdec= (\{P\}\times (\mathcal{Z}\cap \Amdec)) \cup \bigcup_{i} \mathcal{C}_{i}.\end{equation}

Note that 
\begin{equation*}
\begin{aligned}
\dim \mathcal{C}_{i}&=\dim \mathcal{Y}-\dim \mathcal{A}_{g_{v,i}}[n]+ \dim \mathcal{A}_{g_{v,i}-g_{c,j}}[n]\\
&=\dim \mathcal{Y}-\frac{g_{v,i}(g_{v,i}+1)}{2}+ \frac{(g_{v,i}-g_{c,j})(g_{v,i}-g_{c,j}+1)}{2}\\
&=\dim \mathcal{Y}-\frac{1}{2}g_{c,j}(2g_{v,i}+1-g_{c,j})
\end{aligned}
\end{equation*}
Hence the component $\mathcal{C}_{i}$ of $\mathcal{Z}$ is of minimal codimension when $g_{v,i}$ is minimal and $g_{c,j}$ is either minimal or maximal. Hence $\bigcup_{i} \mathcal{C}_{i}$ has codimension in $\mathcal{Y}$ given by
\[\frac{1}{2}\min(g_{c,1}(2g_{v,1}+1-g_{c,1}), g_{c,r}(2g_{v,1}+1-g_{c,r})).\]

The result then follows from the decomposition in \eqref{eq: Yunion} together with Lemma \ref{lem dim1}.

\end{proof}

\subsection{The locus $\Amdec$ and $\mathcal{Z}_L(p,q)[n]$}
For the purpose of our results in Section \ref{section unitary}, we collect here some results on the intersection of $\Amdec$ with Shimura varieties of the form $\mathcal{Z}_L(p,q)[n]$ for some fixed imaginary quadratic field $L$ and positive integers $p$ and $q$.

\begin{lem}\label{lem dim3}
Consider the moduli space $\mathcal{Z}_L(p,q)[n]\subset \mathcal{A}_g[n]$ for some fixed imaginary quadratic field $L$ and positive integers $p$ and $q$. Then $\mathcal{Z}_L(p,q)[n]\cap \Amdec$ has codimension $\min(2p, p+q-2, 2q)$ in $\mathcal{Z}_L(p,q)[n]$.
\end{lem}

\begin{proof}
Suppose $P$ is a point of $\mathcal{Z}_L(p,q)[n]\cap \Amdec$. Then $P\sim Z^2\times P'$ such that $Z^2\in \mathcal{Z}_L(k,\ell)[n]$ and $P'\in \mathcal{Z}_L(p-k,q-\ell)[n]$ for some non-negative integers $k,\ell \ge 0$ such that $k+\ell\ge 2$ is even. 

If $L\hookrightarrow \End_\Q(Z)$, then both $k$ and  $\ell$ are even and the locus of such points in $\mathcal{Z}_L(p,q)[n]\cap \Amdec$ has dimension
\[\dim \mathcal{Z}_L(k/2,\ell/2)[n] + \dim \mathcal{Z}_L(p-k,q-\ell)[n]=pq-p\ell-kq+\frac{5}{4}k\ell.\]
This value is maximal when $k+\ell=2$ in which case the maximal value is $pq-\max(2p, 2q)$. 

If $L$ does not embed in $\End_\Q(Z)$, then $k=\ell$ and the locus of points has dimension
\[\dim \mathcal{A}_{k}[n]\times \dim \mathcal{Z}_L(p-k,q-k)[n]=pq-\frac{1}{2}k(p+q-\frac{3}{2}k-\frac{1}{2}).\]
This value is maximal when $k=1$ in which case its value is $pq-(p+q-2)$. 

It follows that the codimension of $\mathcal{Z}_L(p,q)[n]\cap \Amdec$ in $\mathcal{Z}_L(p,q)[n]$ is $\min(2p, p+q-2, 2q)$.
\end{proof}

\begin{lem}\label{lem dim4}
Fix a product $P=E_1\times \cdots \times E_r$ of pairwise non-isogenous elliptic curves without CM. Consider a subvariety of $\mathcal{A}_g[n]$ of the form
\[\mathcal{Y}=\{P\}\times \mathcal{Z}_L(p,q)[n]\]
for some  fixed imaginary quadratic field $L$ and positive integers $p,q$. Then  $\mathcal{Y}\cap \Amdec$ has codimension $\min(2p, p+q-2, 2q)$ in $\mathcal{Y}$. 
\end{lem}

\begin{proof}
Let $Q\in \mathcal{Y}\cap \Amdec$. Then $Q\sim Z^2\times Q'$ for some abelian varieties $Z, Q'$. Since the $E_i$ are chosen to be pairwise non-isogenous, it follows that $Z^2$ is not isogenous to a subset of the $E_i$. Moreover, if $Z$ is isogenous to some subset of the $E_i$, then $Z\times Q'$ must be isogenous to an element of $\mathcal{Z}_L(p,q)[n]$. But since none of the $E_i$ have CM and none are mutually isogenous, this can only happen if $Z\times Q'$ is isogenous to an element of $\mathcal{Z}_L(p,q)[n]\cap \Amdec$. In other words, writing $Q=P\times R$ where $R\in \mathcal{Z}_L(p,q)[n]$, in this situation we have $R\sim Z\times Q'$ and hence $R\in \mathcal{Z}_L(p,q)[n]\cap \Amdec$. In the case that $Z$ is not isogenous to a subset of the $E_i$, then $Z$ has a simple factor $Z'$ such that $Z'^2$ is isogenous to an element of $\mathcal{Z}_L(p,q)[n]$. Hence once again, writing $Q=P\times R$, we have that $Z'^2$ is a simple factor of $R$ and so $R\in \mathcal{Z}_L(p,q)[n]\cap \Amdec$.

Therefore, the codimension of $\mathcal{Y}\cap \Amdec$  in $\mathcal{Y}$ is in fact given by the codimension of $\mathcal{Z}_L(p,q)[n]\cap \Amdec$ in $\mathcal{Z}_L(p,q)[n]$, t
 By Lemma \ref{lem dim3}, this locus has codimension $\min(2p, p+q-2, 2q)$. 
\end{proof}


\section{Hecke Correspondances and Polarized Isogenies}\label{Hecke section}
\begin{defn}
Let $Z=\Gamma\backslash X^+$ be a connected Shimura variety. Recall from Section \ref{section 1} that this means $\Gamma$ is a congruence subgroup of some reductive $\Q$-algebraic group $G$ such that $G(\mathbb{R})$  acts on $X^+$ by conjugation. 
The \emph{Hecke correspondence} associated to an element $a\in G(\Q)$ consists of  the diagram
\[\begin{tikzcd}
\Gamma\backslash X^+&\arrow{l}[swap]{q}\Gamma_a\backslash X^+\arrow{r}{q_a}&\Gamma\backslash X^+
\end{tikzcd}\]
where 
\begin{itemize}
\item $\Gamma_a=\Gamma \cap a^{-1}\Gamma a$
\item $q(\Gamma_ax)=\Gamma x$
\item$ q_a(\Gamma_a x)=\Gamma ax$
\end{itemize}
In the above diagram, both $q$ and $q_a$ are finite morphisms of degree $[\Gamma:\Gamma_a]$. 

If $Z'$ is a closed irreducible subvariety of $Z$, any irreducible component of $q_a(q^{-1}Z)$ is called a \emph{Hecke translate} of $Z'$, which we denote by $aZ'$. 
\end{defn}

Suppose that $f\colon A\rightarrow B$ is an isogeny between principally polarized abelian varieties $(A,\lambda)$ and $(B,\mu)$. Then the pullback $f^*\mu$ is a polarization on $A$. We say that $f$ is a \emph{polarized isogeny} if there exists an $n\in \mathbb{Z}$ such that $f^*\mu=n\lambda$. For a more thorough treatment of the distinction between isogenies and polarized isogenies see \cite{orr}. 

A point $Q\in \mathcal{A}_g[n]$ is a Hecke translate of another point $P\in \mathcal{A}_g[n]$ if and only if $P$ and $Q$ are related by a polarized isogeny. However the set of polarized isogenies between $P$ and $Q$ can be significantly smaller than the set of isogenies between them. For instance, in \cite[Proposition 3.1]{orr}, Orr shows that the isogeny class of an abelian surface with multiplication by a totally real quadratic field contains infinitely many distinct polarized isogeny classes. 

The goal of this section is to establish conditions in Lemma \ref{polarized isogeny lemma} under which we may conclude that two isogenous abelian varieties are related by a Hecke correspondence. Before doing so, we will establish a bit of notation for abelian varieties and their isogenies.

Recall that if $(A,\lambda)\in \mathcal{A}_g[n]$, the polarization $\lambda$ is  the class of an ample line bundle in $NS(A)$ and that $\nu$ induces an isogeny $\phi_\nu\colon A\rightarrow A^\vee$ of $X$ to its dual abelian variety $A^\vee$. Let us denote
 \[\End_{\Q}(A)\coloneqq \End(A)\otimes_\Z \Q.\]
 Then $\nu$ induces an involution, called the \emph{Rosati involution}, on $\End_\Q(A)$ given by
\begin{equation*}
\begin{aligned}
\End_\Q(A)&\rightarrow \End_\Q(A)\\
f&\mapsto \phi_\lambda^{-1}\circ \hat{f} \circ \phi_\lambda,
\end{aligned}
\end{equation*}
where $\phi_\lambda^{-1}\in \Hom_\Q(A^\vee, A)$ and $\hat{f}$ is the element of $\End_\Q(A^\vee)$ induced by $f$ \cite[pg 114]{BL}.

Let $\End^s A$ (respectively $\End^s_\mathbb{Q}A$) denote the subset of $\End A$ (respectively $\End_\mathbb{Q}X$) of endomorphisms of $A$ invariant under the Rosati involution with respect to $\lambda$ on $\End_\mathbb{Q}A$. Then there is isomorphism of groups 
\begin{equation}\label{eq: delta}\delta\colon NS(A)\rightarrow \End^s A\end{equation}
given by the restriction to $NS(A)$ of the map $ NS_\mathbb{Q}(A)\rightarrow \End^s_\mathbb{Q}A$ defined by $\lambda'\mapsto \phi_\lambda^{-1}\phi_{\lambda'}$ \cite[Proposition 5.2.1]{BL}. 

Lastly, recall that any polarized abelian variety is isogenous to a product of simple polarized abelian varieties and 
that this decomposition is unique up to isogeny \cite[Theorem 4.3.1]{BL}. Hence for any $A\in \mathcal{A}_g[n]$ we may write
\[A\sim \prod_{i=1}^k A_i^{n_i},\]
where the $A_i$ are simple non-isogenous abelian varieties and  the $n_i$,  $\dim A_i$, and $\End_{\Q}(A_i^{n_i})$ are all uniquely determined. We call the $A_i$ the \emph{simple factors} of $A$ (although these are determined only up to isogeny).

\begin{lem}\label{polarized isogeny lemma}
For $(A,\lambda)\in \mathcal{A}_g[n]$, consider the decomposition $A\sim \prod_{i=1}^k A_i^{n_i}$ of $A$ into simple non-isogenous abelian varieties.  Suppose each $n_i=1$ and $\End_\Q(A_i)$ is 
 either $\mathbb{Q}$ or an imaginary quadratic field.  Then every isogeny class of $(A,\lambda)$ is a polarized isogeny class and hence if $(A,\lambda)\sim (B,\mu)$ for some $(B,\mu) \in \mathcal{A}_g[n]$, then there is an element $\alpha \in Sp(2g,\Q)$ such that $\alpha (A,\lambda)=(B,\mu)$.
\end{lem}
 
\begin{proof}
Under the hypotheses of the Lemma, we have that $L=\End_\Q(A)$ decomposes as a product $L=\prod_{i=1}^kL_i$, where each $L_i$ is either $\mathbb{Q}$ or an imaginary quadratic field. In particular the subset of $L$ fixed by the Rosati involution is $L^s=\prod_{i=1}^k\mathbb{Q}$. It follows that the map $\delta$ defined in Equation \eqref{eq: delta} above yields an isomorphism of groups $NS(A)\cong \mathbb{Z}^k$.

Suppose $(B,\mu)\in \mathcal{A}_g[n]$ is isogenous to $(A,\lambda)$ and let $f\colon A\rightarrow B$ be an isogeny. Then $f^*\mu$ is a polarization of $A$ and so may be viewed as a tuple of integers $\prod_{i=1}^k n_i$. Precomposing the isogeny $f$ with the necessary elements of $\End A_i$, we may scale as needed to ensure that all of the integers $n_i$ are equal. This new isogeny $f'\colon A\rightarrow B$ is thus a polarized isogeny. 
\end{proof}


\section{The subset $\Gamma\subset Sp(2g,\Q)$ and $\Adec$}\label{sec: Gamma}
Let us fix a Shimura variety $\mathcal{Z}=\mathcal{A}_{g_1}[n]\times \cdots \times \mathcal{A}_{g_r}[n]$
in  $\mathcal{A}_g[n]$,
where  $\sum_{i=1}^rg_i=g$. The main result of this section will be Corollary \ref{main cor}, which shows that if $\alpha \in Sp(2g,\Q)$ is general, more precisely if $\alpha$ lies outside of some codimension $2$ subvariety of $Sp(2g,\Q)$, then $\alpha \mathcal{Z}\cap \Adec\subset \Amdec$. The above codimension $2$ subvariety will be defined in terms of elements of $Sp(2g,\Q)$ which send a fixed product of elliptic curves into the decomposable locus $\Adec$.

Let $\mathcal{P}$ denote the set of proper partitions of the set $\{1, \ldots, g\}$. For each $\lambda\in \mathcal{P}$, we may write $\lambda$ in terms of its components as $\lambda=(\lambda_1,\ldots, \lambda_s)$. The term ``proper'' then just excludes the partition  $(1\cdots g)\coloneqq (\{1\},\ldots, \{g\})$. 
For any partition $\lambda=(\lambda_1,\ldots, \lambda_s)$ in $\mathcal{P}$, let $\ell_i(\lambda)$ denote the length of the component $\lambda_i$. Hence for any $\lambda \in \mathcal{P}$ we have $\sum_{i=1}^s\ell_i(\lambda)=g$. 

Then define for any $\lambda\in \mathcal{P}$ the set
\[\Adeclam=\{ A=A_1\times \cdots \times A_s \in \Adec\mid \dim A_i=\ell_i(\lambda) \ \forall 1\le i\le s\}.\]
Note that 
\[\Adec=\bigcup_{\lambda\in \mathcal{P}}\Adeclam.\]

Let $\lambda_\mathcal{Z}$ be the partition corresponding to the Shimura variety $\mathcal{Z}$, so that $\ell_i(\lambda_{\mathcal{Z}})=g_i$ and
\[\mathcal{Z}=\mathcal{A}_{g,\lambda_\mathcal{Z}}^{\mathrm{dec}}[n].\]

Observe that if $\delta\in \mathcal{P}$ is a refinement of the partition $\lambda_{\mathcal{Z}}$, then the Shimura variety $\mathcal{Z}$ contains $\mathcal{A}_{g,\delta}^{\mathrm{dec}}[n]$ but also contains many other irreducible subvarieties isomorphic to $\mathcal{A}_{g,\delta}^{\mathrm{dec}}[n]$. Each of these subvarieties is itself a Shimura variety parametrizing principally polarized abelian varieties with a fixed decomposition. For instance consider the case $g=3$ and $\lambda_{\mathcal{Z}}=(1,23)$, so that $\mathcal{Z}= \mathcal{A}_1[n]\times \mathcal{A}_2[n]$. 
If $\delta=(1,2,3)$, then $\mathcal{Z}$ contains $\mathcal{A}_{3,\delta}^{\mathrm{dec}}[n]= \mathcal{A}_1[n]\times \mathcal{A}_1[n]\times \mathcal{A}_1[n]$ parametrizing threefolds which decompose as  a product of elliptic curves $E_1\times E_2\times E_3$, but also contains many other irreducible subvarieties isomorphic to $\mathcal{A}_{3,\delta}^{\mathrm{dec}}[n]$ whose points are abelian threefolds of the form $E_1\times S$, where $S$ is an indecomposable abelian surface isogenous to a product of elliptic curves $E_2\times E_3$. 

In the lemma below, we wish to understand when a Hecke translate of such a subvariety isomorphic to $\mathcal{A}_{g,\delta}^{\mathrm{dec}}[n]$ is contained in the decomposable locus $\Adec$.

\begin{lem}\label{lem: Shimura Hecke}
Let $\delta=(\delta_1,\ldots, \delta_t)\in \mathcal{P}$ and let $\mathcal{Y}_\delta \cong \mathcal{A}_{g,\delta}^{\mathrm{dec}}[n]$ be a Shimura variety such that each point of $\mathcal{Y}_\delta$ is isogenous to a point in $\mathcal{A}_{g,\delta}^{\mathrm{dec}}[n]$. Suppose $P\in \mathcal{Y}_\delta$ satisfies $P\sim A_1\times \cdots \times A_t\in \mathcal{A}_{g,\delta}^{\mathrm{dec}}[n]$ with each $A_i\in \mathcal{A}_{\ell_i(\delta)}$ simple. 
If $\alpha\in Sp(2g,\Q)$ is such that $\alpha P\in \Adeclam$ for some partition $\lambda\in \mathcal{P}$, then $\alpha \mathcal{Y}_\delta \subset \Adeclam$
\end{lem}

\begin{proof}
By the definition of $\mathcal{Y}_\delta$, we know $\mathcal{Y}_\delta \cong \prod_{i=1}^t\mathcal{A}_{\ell_i(\delta)}[n]$. Recall from Section \ref{section 1} that each $\mathcal{A}_{\ell_i(\delta)}[n]$ is the quotient $\ker(Sp(2\ell_i(\delta),\Z)\rightarrow Sp(2\ell_i(\delta), \Z/n\Z))\backslash\mathbb{H}_{\ell_i(\delta)}$. Each of the Siegel upper half-spaces $\mathbb{H}_{\ell_i(\delta)}$ is the Hermitian symmetric space $Sp(2\ell_i(\delta), \mathbb{R})/U(\ell_i(\delta))$ and thus, in particular, is equipped with a transitive action by $Sp(2\ell_i(\delta), \mathbb{R})$. 
Hence both $\mathcal{Y}_\delta$ and its Hecke translate $\alpha \mathcal{Y}_\delta$ are equipped with a transitive action by the group $\prod_{i=1}^tSp(2\ell_i(\delta), \mathbb{R})$. 

Thus let $Q\in  \mathcal{Y}_\delta$ and consider $\alpha Q \in \alpha \mathcal{Y}_\delta$. Then via the above transitive action on $\alpha  \mathcal{Y}_\delta$, there exists an element $M\in \prod_{i=1}^tSp(2\ell_i(\delta), \mathbb{R})$ such that $M \alpha P=\alpha Q$. 

However, we also know that $\alpha P\in \Adeclam$. This $\Adeclam$ is isomorphic to $\prod_{j=1}^s\mathcal{A}_{\ell_j(\lambda)}[n]$ and thus is equipped with a transitive action by the group $\prod_{j=1}^sSp(2\ell_j(\lambda), \mathbb{R})$.

Moreover, observe that since $P\sim A_1\times \cdots \times A_t\in \mathcal{A}_{g,\delta}^{\mathrm{dec}}[n]$ with each $A_i\in \mathcal{A}_{\ell_i(\delta)}$ simple and $\alpha P\in \Adeclam$, it follows that the partition $\delta$ is a refinement of the partition $\lambda$ and that this refinement induces a natural  embedding $\prod_{i=1}^tSp(2\ell_i(\delta), \mathbb{R})\hookrightarrow \prod_{j=1}^sSp(2\ell_j(\lambda), \mathbb{R})$ compatible with the respective group actions on $\alpha \mathcal{Y}_\delta$ and $\Adeclam$. Thus, via this embedding, it follows that $M$ is an element of $\prod_{j=1}^sSp(2\ell_j(\lambda), \mathbb{R})$ acting on $\Adeclam=\prod_{j=1}^sSp(2\ell_j(\lambda), \mathbb{R})/\prod_{j=1}^sU(\ell_j(\lambda))$. Then since $\alpha P \in \Adeclam$, it follows that $M\alpha P=\alpha Q$ is an element of $\Adeclam$. We have thus shown that indeed $\alpha  \mathcal{Y}_\delta\subset \Adeclam$.

\end{proof}

Now let us fix a product of pairwise non-isogenous elliptic curves in $\mathcal{A}_g[n]$
\[P_0=E_1\times \cdots \times E_g.\]
In particular,  point $P_0$ of $\mathcal{A}_g[n]$ is a point of the Shimura variety $\mathcal{Z}$.  Consider the subset of $Sp(2g,\Q)$ given by
\[\Gamma=\{\alpha \in Sp(2g,\Q)\mid \alpha P_0\in \Adec\}.\]
For any partition $\lambda \in \mathcal{P}$, define a subset of $\Gamma$ by
\[\Gamma_\lambda=\{\alpha\in \Gamma \mid \alpha P_0\in \Adeclam\}.\]
It then follows that 
\[\Gamma=\bigcup_{\lambda\in \mathcal{P}} \Gamma_\lambda.\]

Recall that for any $g\ge 1$, the Siegel upper half-space $\mathbb{H}_g$ is equipped with a transitive action by $Sp(2g,\mathbb{R})$ such that the stabilizer of a point is $U(g,\mathbb{R})$, so that $\mathbb{H}_g=Sp(2g,\mathbb{R})/U(g,\mathbb{R})$. 

\begin{lem}\label{group cor}
For any partition $\lambda=(\lambda_1,\ldots, \lambda_s)$ in $\mathcal{P}$, there is an isomorphism 
\[\Gamma_\lambda\cong 
\left(\prod_{i=1}^s Sp(2\ell_i(\lambda),\mathbb{Q})\right)\cdot 
U(g,\Q).\]
\end{lem}

\begin{proof}
Observe that for the point $P_0= E_1\times \cdots \times E_g$, if $\alpha \in \prod_{i=1}^s Sp(2\ell_i(\lambda),\mathbb{Q})$, then $\alpha P_0\in \Adeclam$. Hence for any $\alpha \in \prod_{i=1}^s Sp(2\ell_i,\mathbb{Q})$ and $u\in U(g,\mathbb{Q})$, the point $\alpha u P_0=\alpha P_0$ lies in $\Adeclam$. It follows that there is an embedding
\[f\colon\left(\prod_{i=1}^s Sp(2\ell_i(\lambda),\mathbb{Q})\right)\cdot U(g,\mathbb{Q})\rightarrow \Gamma_\lambda\]
given by the product of the natural embeddings $\prod_{i=1}^s Sp(2\ell_i(\lambda),\mathbb{Q})\hookrightarrow Sp(2g,\Q)$ and $U(g,\Q)\hookrightarrow Sp(2g,\Q)$.
To see that $f$ is surjective, note that by Lemma \ref{polarized isogeny lemma} together with the fact that all simple factors of $P_0$ have dimension $1$, for any $\alpha'\in \Gamma_\lambda$ there exists an $\alpha \in \prod_{i=1}^s Sp(2\ell_i(\lambda),\mathbb{Q})$ such that $\alpha P_0=\alpha'P_0$. Hence $\alpha^{-1}\alpha'$ lies in the stabilizer of the period point of $P_0$. Since $\alpha^{-1}$ and $\alpha'$ are both in $Sp(2g,\Q)$, it follows that  $\alpha^{-1}\alpha'\in U(g,\Q)$. Namely, there exists $u\in U(g,\Q)$ such that $\alpha'=\alpha u$ and thus $\alpha' \in \left(\prod_{i=1}^s Sp(2\ell_i(\lambda),\mathbb{Q})\right)\cdot U(g,\Q)$.

\end{proof}

\begin{lem} \label{main lem}
Consider a point $R=A_1\times \cdots \times A_r$ in $\mathcal{Z}-\Amdec$, where each $A_i\in \mathcal{A}_{g_i}[n]$. 
If $\alpha\in Sp(2g,\Q)$ 
is such that $\alpha R\in \mathcal{A}_{g,\delta}^{\mathrm{dec}}[n]$ for some $\delta\in \mathcal{P}$, 
then $\alpha \in \Gamma_\delta \cdot \prod_{i=1}^rSp(2\ell_i(\lambda_{\mathcal{Z}}),\Q)$.
\end{lem}

\begin{proof}
For each $A_i$ in the decomposition $R=A_1\times \cdots \times A_r$ we may write $A_i\sim \prod_{j=1}^{m_i}Y_{i,j}$, where all of the $Y_{i,j}$ are simple and pairwise non-isogenous. Let $g_{i,j}=\dim Y_{i,j}$ and let $\mathcal{Y}\subset \mathcal{Z}$ be the Shimura variety through $R$ parametrizing abelian varieties isogenous to a product $\prod_{i,j}Y'_{i,j}$, where $\dim Y'_{i,j}=g_{i,j}$ for all $i$ and $j$. 

Since $R\notin \Amdec$, the point $P_0$ is isogenous to a point of $\mathcal{Y}$. Hence by Lemma \ref{polarized isogeny lemma} there exists an element $\beta \in \prod_{i=1}^rSp(2\ell_i(\lambda_{\mathcal{Z}}),\Q)$ such that $\beta P_0 \in \mathcal{Y}$. Then since $\alpha R\in \mathcal{A}_{g,\delta}^{\mathrm{dec}}[n]$, Lemma \ref{lem: Shimura Hecke} implies that $\alpha \beta P_0\in \mathcal{A}_{g,\delta}^{\mathrm{dec}}[n]$. Hence $\alpha \beta\in \Gamma_\delta$. It follows that $\alpha \in \Gamma_\delta \cdot \prod_{i=1}^rSp(2\ell_i(\lambda_{\mathcal{Z}}),\Q)$.
\end{proof}

\begin{rem}
Note that the proof of Lemma \ref{main lem} crucially uses the assumption that $R\notin \Amdec$ to conclude that $P_0$ is isogenous to a point of $\mathcal{Y}$. For instance, suppose that $R$ was isogenous to a product $E_u\times E_v^2$ for some elliptic curves $E_u$ and $E_v$. Then although the point $R$ lies on a Shimura variety isomorphic to $\mathcal{A}_1[n]\times \mathcal{A}_1[n]$ whose points $P$ all have the form $P\sim E_x\times E_y^2$ with $E_x$ and $E_y$ elliptic curves, there is not necessarily a Shimura variety passing through $R$ whose general point is isogenous to a product of the form  $E_x\times E_y\times E_z$ where $E_x$, $E_y$, and $E_z$ are pairwise non-isogenous elliptic curves. 
\end{rem}

Lemma \ref{main lem} implies that if $\alpha \in Sp(2g,\Q)$ is not contained in any $\Gamma_\delta \cdot \prod_{i=1}^rSp(2\ell_i(\lambda_{\mathcal{Z}}),\Q)$ for $\delta \in \mathcal{P}$,  then $\alpha \mathcal{Z}\cap \Adec \subset \Amdec$. Moreover, we have already established in Lemma \ref{lem dim1} that $\mathcal{Z}\cap \Amdec$ has high codimension in $\mathcal{Z}$ and hence that the same is true of $\alpha\mathcal{Z}\cap \Amdec$ in $\alpha \mathcal{Z}$. 

Thus it remains to verify that such $\alpha \in Sp(2g,\Q)$ exist. To show this we will compute the $\Q$-dimension of 
\[
\Gamma\cdot \prod_{i=1}^rSp(2\ell_i(\lambda_{\mathcal{Z}}),\Q)=\bigcup_{\delta \in \mathcal{P}}\Gamma_\delta \cdot \prod_{i=1}^rSp(2\ell_i(\lambda_{\mathcal{Z}}),\Q)
\]
 as a subvariety of $Sp(2g,\Q)$. 

\begin{lem} \label{dim lemma}
For the partitions $\delta\in \mathcal{P}$ we have
 \[\max_{\delta\in \mathcal{P}}\dim_\Q \left(\Gamma_\delta \cdot \prod_{i=1}^rSp(2\ell_i(\lambda_{\mathcal{Z}}),\Q)\right)=
 2g^2+g-2.\] 
\end{lem}

\begin{proof}
We proceed by induction on $g\ge 2$. If $g=2$, then $\delta=\lambda_{\mathcal{Z}}=(1,1)$. Hence $\Gamma_\delta \cdot \prod_{i=1}^rSp(2\ell_i(\lambda_{\mathcal{Z}}),\Q)=(SL(2,\Q)\times SL(2,\Q))\cdot U(2,\Q)$,  has $\Q$-dimension 
$3+3+4-2=8$, which is indeed equal to $2\cdot 2^2+2-2$.

Now consider partitions  $\lambda_{\mathcal{Z}}=(\lambda_1,\ldots, \lambda_r)$ and $\delta=(\delta_1,\ldots, \delta_t)$ of the set $\{1,\ldots, g\}$. Observe that $\lambda_{\mathcal{Z}}$ is obtained from a partition $\lambda_{\mathcal{Z}}'$ of the set $\{1,\ldots, g-1\}$ by adjoining the element $g$ to the component of $\lambda_{\mathcal{Z}}'$ in the $j$-th position, for some $1\le j\le r$. Similarly $\delta$ is obtained from a partition $\delta'$ of the set $\{1,\ldots, g-1\}$ by adjoining the element $g$ to the component of $\delta'$ in the $k$-th position, for some $1\le k\le t$ (where we may have $\ell_r(\lambda_{\mathcal{Z}}')=0$ or $\ell_t(\delta')=0$). 

Note that for any $m\ge 1$, we have $\dim_\Q Sp(2(m+1),\Q)-\dim_\Q Sp(2m,\Q)=4m+3$ and $\dim_\Q U(m+1,\Q)-\dim_\Q U(m,\Q)=2m+1$. Hence 
\begin{equation}\label{eq: dim prod Sp}\dim_\Q \left( \prod_{i=1}^rSp(2\ell_i(\lambda_{\mathcal{Z}}),\Q)\right)=\dim_\Q\left( \prod_{i=1}^rSp(2\ell_i(\lambda_{\mathcal{Z}}'),\Q)\right) + 4\ell_j(\lambda_{\mathcal{Z}}')+3
\end{equation}
and it follows from Lemma \ref{group cor} that

\begin{equation}
\begin{aligned}\label{eq: dim Gamma}
\dim_\Q\Gamma_\delta&=\dim_\Q \left( \prod_{i=1}^tSp(2\ell_i(\delta),\Q)\right) + \dim_\Q U(g,\Q) -\dim_\Q \left(\prod_{i=1}^tU(\ell_i(\delta),\Q)\right)\\
&= \dim_\Q\left(\prod_{i=1}^tSp(2\ell_i(\delta'),\Q)\right) + 4\ell_k(\delta')+3+ g^2-\dim_\Q\left( \prod_{i=1}^tU(\ell_i(\delta'),\Q)\right)-2\ell_k(\delta')-1\\
&= \dim_\Q\left(\prod_{i=1}^tSp(2\ell_i(\delta'),\Q)\right) + (g-1)^2-\dim_\Q\left( \prod_{i=1}^tU(\ell_i(\delta'),\Q)\right)+2\ell_k(\delta')+2g+1\\
&=\dim_\Q\Gamma_{\delta'} + 2\ell_k(\delta')+2g+1.
\end{aligned}
\end{equation}

Now let $D_{jk}$ denote the set of elements of $\{1,\ldots, g-1\}$ occurring in both the $j$-th position of $\lambda_{\mathcal{Z}}'$ and the $k$-th position of $\delta'$. Let $d_{jk}$ be the cardinality of the set $D_{jk}$. Then since the element $g$ is added to both the $j$-th position of $\lambda_{\mathcal{Z}}'$ and the $k$-th position of $\delta'$ when forming $\lambda_{\mathcal{Z}}$ and $\delta$, we obtain
\begin{equation}\label{eq: dim intersection}
\begin{aligned}
\dim_\Q \left(\Gamma_\delta\cap \prod_{i=1}^rSp(2\ell_i(\lambda_{\mathcal{Z}}),\Q)
 \right)
 &=&\left(\dim_\Q\left(\prod_{i=1}^tSp(2\ell_i(\delta'),\Q)\cap \prod_{i=1}^rSp(2\ell_i(\lambda_{\mathcal{Z}}'),\Q)\right)+4d_{jk}+3\right)\\
  & &+ \left(\dim_\Q \prod_{i=1}^rU(\ell_i(\lambda_{\mathcal{Z}}',\Q) + 2\ell_j(\lambda_{\mathcal{Z}}')+1\right)\\
  & & - \left(\dim_\Q\left(\prod_{i=1}^tU(\ell_i(\delta'),\Q)\cap \prod_{i=1}^rSp(2\ell_i(\lambda_{\mathcal{Z}}'),\Q)\right)+2d_{jk}+1\right)\\
 &=&\dim_\Q\left(\Gamma_{\delta'}\cap  \prod_{i=1}^rSp(2\ell_i(\lambda_{\mathcal{Z}}'),\Q) \right)+ 2d_{jk} + 2\ell_j(\lambda_{\mathcal{Z}}')+ 3
 \end{aligned}
 \end{equation}
 Combining \eqref{eq: dim prod Sp}, \eqref{eq: dim Gamma}, and \eqref{eq: dim intersection}, it follows that
 \[
 \begin{aligned}
 \dim_\Q \left( \Gamma_\delta\cdot \prod_{i=1}^rSp(2\ell_i(\lambda_{\mathcal{Z}}),\Q)\right)&=&\dim_\Q\Gamma_{\delta'} + 2\ell_k(\delta')+2g+1\\
  & & + \dim_\Q\left( \prod_{i=1}^rSp(2\ell_i(\lambda_{\mathcal{Z}}'),\Q)\right) + 4\ell_j(\lambda_{\mathcal{Z}}')+3\\
  & & -\dim_\Q\left(\Gamma_{\delta'}\cap  \prod_{i=1}^rSp(2\ell_i(\lambda_{\mathcal{Z}}'),\Q) \right)+ 2d_{jk} + 2\ell_j(\lambda_{\mathcal{Z}}')+ 3\\
  &=& \dim_\Q \left(\Gamma_{\delta'}\cdot  \prod_{i=1}^rSp(2\ell_i(\lambda_{\mathcal{Z}}'),\Q)\right) + 2(g+\ell_j(\lambda_{\mathcal{Z}}') + \ell_k(\delta')-d_{jk})+1
   \end{aligned}
 \]

 By the pigeonhole principle, the maximum value across all $\delta\in \mathcal{P}$ that $\ell_j(\lambda_{\mathcal{Z}}') + \ell_k(\delta')-d_{jk}$ attains is $g-1$. Thus by induction we obtain
\begin{equation*}
\begin{aligned}
 \max_{\delta\in \mathcal{P}}\dim_\Q \left( \Gamma_\delta\cdot \prod_{i=1}^rSp(2\ell_i(\lambda_{\mathcal{Z}}),\Q)\right)&=\max_{\delta\in \mathcal{P}} \dim_\Q \left(\Gamma_{\delta'}\cdot  \prod_{i=1}^rSp(2\ell_i(\lambda_{\mathcal{Z}}'),\Q)\right) + 2(g+g-1)+1
\\
&= 2(g-1)^2+(g-1)-2+4g-1\\
&=2g^2+g-2.
\end{aligned}
\end{equation*}

\end{proof}

In particular, we obtain as a consequence

\begin{cor}\label{main cor}
Fix a Shimura variety $\mathcal{Z}=\mathcal{A}_{g_1}[n]\times \cdots \times \mathcal{A}_{g_r}[n]$
in  $\mathcal{A}_g[n]$,
where  $\sum_{i=1}^rg_i=g$. Then for any general $\alpha\in Sp(2g,\Q)$ we have
\[\alpha\mathcal{Z}\cap \Adec \subset \Amdec.\]
\end{cor}

\begin{proof}
Let $R\in \mathcal{Z}$ such that $\alpha R\in \Adec$. 
 If $R$ is not contained in $ \Amdec$, then by Lemma \ref{main lem} we have $\alpha \in \Gamma \cdot \prod_{i=1}^rSp(2\ell_i(\lambda_{\mathcal{Z}}),\Q)$. However since $\dim_\Q Sp(2g,\Q)=2g^2+g$, it follows from Lemma \ref{dim lemma} that $ \Gamma \cdot \prod_{i=1}^rSp(2\ell_i(\lambda_{\mathcal{Z}}),\Q)$ is a proper subvariety of $Sp(2g,\Q)$ of codimension at least $2$, which contradicts the assumption that $\alpha$ is general. Therefore, we must have $R\in  \Amdec$ and so $\alpha R\in \Amdec$.
\end{proof}


\section{The Mumford-Tate group and connected monodromy}\label{section MT}
A $\Z$- (respectively $\Q$-) Hodge structure $V$ is an abelian group (respectively a $\Q$-vector space) such that $V_\C\coloneqq V\otimes_\Z \C$ (respectively $V_\C\coloneqq V\otimes_\Q \C$) comes 
equipped with a decomposition $V_\C=\oplus_{p,q}V^{p,q}$ such that $\overline{V}^{q,p}=V^{p,q}$. This bigrading determines and is determined by a homomorphism 
\[h\colon \mathbb{C}^*\rightarrow GL(V_\mathbb{R})\] given by $h(\lambda)v=\sum_{p,q}\lambda^q\overline{\lambda}^pv^{p,q}$. 

 The \emph{Mumford-Tate group} $MT(V)\subset GL(V)$ of either a $\Z$- or $\Q$-Hodge structure $V$ is the smallest $\Q$-algebraic subgroup of $GL(V)$ whose real points contain the image of $\mathbb{C}^*$ under the homomorphism $h$. 
It is sometimes more convenient to work with the \emph{special Mumford-Tate group} $SMT(V)$ (also known as the \emph{Hodge group} $Hg(V)$) which is the identity component of $SL(V)\cap MT(V)$. In other words $SMT(V)$ is the smallest $\Q$-algebraic  subgroup of $SL(V)$ containing the real points of the image of the unit circle $U(1)$. 
For more on Mumford-Tate groups see for instance \cite{moonenmt}.

Recall that the \emph{derived group} $\mathcal{D}G$ of an algebraic group $G$ over a field $k$ is the intersection of the normal algebraic subgroups $N$ of $G$ such that $G/N$ is commutative. In what follows, we will consider the derived group $\mathcal{D}MT(V)$ of the Mumford-Tate group of a Hodge structure $V$.  

Given a variation of polarizable $\Z$-  (respectively  $\Q$-) Hodge structures $\mathbb{V}$ over a connected algebraic variety $B$, every fiber $\mathbb{V}_b$ is a $\Z$- (respectively $\Q$-) Hodge structure and thus has a Mumford-Tate group $MT(\mathbb{V}_b)$. Moreover, for $b\in B$ very general consider the monodromy representation
\[\Phi\colon \pi_1(B,b)\rightarrow  GL(\mathbb{V}_b\otimes \Q).\]

\begin{defn}The \emph{connected monodromy group} $T(\mathbb{V})$ of $\mathbb{V}$ is the connected component of the identity of the $\Q$-Zariski closure of the image of $\Phi$. In other words $T$ is the smallest $\Q$-algebraic subgroup of $GL(\mathbb{V}_b)$ containing the image of $\Phi$.
\end{defn}

Note that there is an equivalence of categories between the category of complex abelian varieties and the category of polarizable torsion-free $\Z$-Hodge structures of type $(-1,0), (0,-1)$, meaning that $V$ satisfies $V_\C\cong V^{1,0}\oplus V^{0,1}$. This equivalence is given by the functor sending a complex abelian variety $A$ to $H^1(A,\Z)$ (See for instance \cite[Section 2.1]{moonenmt}). This equivalence then extends via $A\mapsto H^1(A,\Q)$ to an equivalence of categories between the category of complex abelian varieties up to isogeny and the category of polarizable $\Q$-Hodge structures of type $(-1,), (0,-1)$.
Hence suppose  $f\colon \mathscr{A}\rightarrow B$ is a family of complex abelian varieties on the connected algebraic variety $B$ and consider the variation of $\Q$-Hodge structures $\mathbb{V}$ given by $\mathbb{V}_b=H^1(A_b,\Q)$ for every $b\in B$, where $A_b=f^{-1}(b)$. We will then refer to the \emph{connected monodromy} $T(\mathscr{A})$ of the family $\mathscr{A}$ to mean the connected monodromy $T(\mathbb{V})$ of the variation of Hodge structures $\mathbb{V}$. 

Since the connected monodromy group only depends on the $\Q$-structure of the underlying variation of Hodge structures, for the purposes of computing connected monodromy we will only concern ourselves with $\Q$-variations of Hodge structures. However the above equivalences of categories are helpful in thinking about the distinction between a complex abelian variety being decomposable (meaning that its corresponding $\Z$-Hodge structure decomposes) and being non-simple (meaning that its corresponding $\Q$-Hodge structure decomposes). 

Now for any variation of $\Q$-Hodge structures $\mathbb{V}$ on a connected algebraic variety $B$, the Theorem of the Fixed Part (see \cite[Theorem 4.1.1]{deligneii}) implies that the monodromy-invariant classes $\mathbb{V}_b^\Phi$ on a single fiber are restrictions of classes on the total space and thus these classes arise from a sub-variation of Hodge structures $\mathbb{V}^\Phi$ on $B$. Since the category of polarizable variations of $\Q$-Hodge structures on a compactifiable base is semisimple, we then obtain a decomposition
\[\mathbb{V}=\mathbb{V}^\Phi \oplus \mathbb{V}_{\mathrm{var}},\]
where $\mathbb{V}_{\mathrm{var}}$ denotes the variation of $\Q$-Hodge structures on $B$ giving $\mathbb{V}_b/\mathbb{V}^\Phi_b$ on every $b\in B$. 
 Since the base $B$ of the variation $\mathbb{V}$ is a connected algebraic variety, the theorem of Andr\'e \cite[Theorem 1]{andre} then implies that, for $b\in B$ very general,  the connected monodromy group $T(\mathbb{V})$ satisfies
\begin{equation}\label{eq: andre}T(\mathbb{V}) \triangleleft \mathcal{D}MT((\mathbb{V}_{\mathrm{var}})_b),\end{equation}
where $\mathcal{D}MT((\mathbb{V}_{\mathrm{var}})_b)$ denotes the derived group of the fiber at $b$ of the $\Q$-variation of Hodge structures $\mathbb{V}_{\mathrm{var}}$. From this we obtain:

 \begin{prop}\label{bb prop}
Let $\mathcal{Z}$ be a subvariety of  $\mathcal{A}_g[n]$ and let $\mathcal{Y}\subset \mathcal{Z}$ be a subvariety of codimension at least $d\ge2$. Let $\mathbb{V}$ be the variation of $\Z$-Hodge structures on $\mathcal{Z}$ given by $V_z=H^1(A_z, \Z)$, where $A_z$ denotes the principally polarized abelian variety determined by $z\in \mathcal{Z}$. 

If the boundary $\mathcal{Z}^*\backslash \mathcal{Z}$ of a compactification $\mathcal{Z}^*$ of $\mathcal{Z}$ has codimension at least $d$ in $\mathcal{Z}^*$, then any general complete intersection of ample divisors on $\mathcal{Z}^*$ yields a $(d-1)$-dimensional smooth compact subvariety $S$ of $\mathcal{Z}\backslash \mathcal{Y}$ such that, for $s\in S$ very general, the connected monodromy group of the universal family of $S$ is a normal subgroup of $ \mathcal{D}MT((\mathbb{V}_{\mathrm{var}})_s)$. 
\end{prop}

\begin{proof}
Since the boundary $\mathcal{Z}^*\backslash \mathcal{Z}$ of $\mathcal{Z}^*$ has codimension at least $d$ in $\mathcal{Z}^*$, it follows that a $(d-1)$-dimensional  subvariety $S$ of $\mathcal{Z}^*$ obtained as a 
 general complete intersection of ample divisors will avoid  $\mathcal{Z}^*\backslash \mathcal{Z}$ and thus lie entirely in $Z$. Moreover, since $\mathcal{Y}\subset \mathcal{Z}$ has codimension at least $d$ in $\mathcal{Z}$, we also have $\mathcal{Y}\cap S=\emptyset$. Thus indeed such an $S$ is a $(d-1)$-dimensional compact subvariety of $\mathcal{Z}\backslash \mathcal{Y}$.
 
Consider the universal family $f\colon\mathscr{A}\rightarrow S$ of abelian varieties over $S\subset \mathcal{A}_g[n]$. Let $\mathbb{W}=\mathbb{V}|_S$ denote the induced variation of $\Z$-Hodge structures on $S$ given by $s\mapsto H^1(A_s, \Z)$.  Leting $\Psi$ denote the monodromy representation attached to $\mathbb{W}$ at a very general point $s\in S$, we have $\mathbb{W}=\mathbb{W}^\Psi\oplus \mathbb{W}_{\mathrm{var}}$ and, by the 
 theorem of Andr\'e \cite[Theorem 1]{andre} stated above in \eqref{eq: andre}, the connected monodromy group $T(\mathscr{A})$ is a normal subgroup of $\mathcal{D}MT((\mathbb{W}_{\mathrm{var}})_s)$.
  
 However, by construction, since the point $s$ is very general in $S$ and $S$ is obtained as a general complete intersection in $\mathcal{Z}$, we have that $s$ is also very general in $\mathcal{Z}$. Moreover if $\Phi$ is the monodromy representation of the variation of $\mathbb{V}$ at the point $s$, the Theorem of the Fixed Part (see \cite[Theorem 4.1.1]{deligneii}) implies that $\mathbb{V}^\Phi|_S=\mathbb{W}^{\Psi}$ and $\mathbb{V}_{\mathrm{var}}|_S=\mathbb{W}_{\mathrm{var}}$. Therefore $\mathcal{D}MT((\mathbb{W}_{\mathrm{var}})_s)=\mathcal{D}MT((\mathbb{V}_{\mathrm{var}})_s)$ and so $T(\mathscr{A})$ is a normal subgroup of $\mathcal{D}MT((\mathbb{V}_{\mathrm{var}})_s)$.
 
 \end{proof}


\section{Complete Families of Abelian Varieties with Fixed Decomposition Type}\label{section main}
With the tools described in Sections \ref{section 1}--\ref{section MT}  in place, we are now ready to present the proof of Theorem \ref{main theorem}, which we restate here for convenience.

\begin{thm}
Consider any fixed product of pairwise non-isogenous abelian varieties $A_{c,1}\times \cdots \times A_{c,r}\times A_{v,1}\times \cdots \times A_{v,s}$ such that each $A_{v,j}$ is general in the moduli space $\mathcal{A}_{g_{v,j}}$ and all $g_{v,j}\ge 2$. Then for any positive integer $d\le \min_j g_{v,j}-1$, there exist infinitely many complete families  $f\colon \mathscr{A}\rightarrow B$ of maximal variation of $g$-dimensional principally polarized abelian varieties with the properties:
\begin{enumerate}
\item The base $B$ is a smooth projective variety of dimension $d$
\item Every fiber of $f$ is indecomposable.
\item If $A$ is a very general fiber of $f$, then $A$ is isogenous to the product
\[A_{c,1}\times \cdots \times A_{c,r}\times A_{v,1}\times \cdots \times A_{v,s}.\]
\item The connected monodromy group of $f\colon \mathscr{A}\rightarrow B$ is 
\[T(\mathscr{A})=\prod_{j=1}^{s}Sp_\Q(V_{v,j}),\]
where $V_{v,j}=H^1(A_{v,j},\Q)$ for each $1\le j\le s$.

\end{enumerate}
\end{thm}

\begin{proof}
Let us suppose without loss of generality that $g_{c,1}\le \cdots \le g_{c,r}$ and $g_{v,1}\le \cdots \le g_{v,s}$. 
For some fixed $n\ge 3$, consider the product of moduli spaces
\[\mathcal{Z}=\mathcal{A}_{g_{c,1}}[n]\times \cdots \times \mathcal{A}_{g_{c,r}}[n]\times \mathcal{A}_{g_{v,1}}[n]\times \cdots \times \mathcal{A}_{g_{v,s}}[n].\]

Let $\alpha$ be any general element of $Sp(2g,\Q)$. Then by Corollary \ref{main cor}, we have 
\begin{equation}\label{eq: Z}\alpha \mathcal{Z}\cap \Adec \subset \Amdec.\end{equation}

Fixing a level $n$ structure on the chosen abelian varieties $A_{c,1}\times \cdots \times A_{c,r}$, consider the subvariety of $\mathcal{Z}$ given by
\[\mathcal{Y}=\{A_{c,1}\}\times \cdots \times \{A_{c,r}\}\times \mathcal{A}_{g_{v,1}}[n]\times \cdots \times \mathcal{A}_{g_{v,s}}[n].\]
It then follows from \eqref{eq: Z} that 
\[\alpha \mathcal{Y}\cap \Adec \subset \Amdec.\]

Hence if  $P\in \mathcal{Y}$ satisfies $\alpha P\in  \Adec$, then $\alpha P\in  \Amdec$ and so $P\in \Amdec$. Moreover, it follows from Lemma \ref{lem dim2}, that 
the subvariety $\mathcal{Y}\cap \Amdec$ of $\mathcal{Y}$ has codimension in $\mathcal{Y}$
\begin{equation}\label{eq: minimum}\min(2g_{v,1}-2,\frac{1}{2}g_{c,1}(2g_{v,1}+1-g_{c,1}), \frac{1}{2}g_{c,r}(2g_{v,1}+1-g_{c,r})).\end{equation}
Note that for any $1\le j\le s$ and any indeterminate integer $x$, the value
 $ \frac{1}{2}x(2g_{v,j}+1-x)$ achieves its minimum when $x=1$ or $x=2g_{v,j}$ in which case this minimum value is $g_{v,j}$. Since $g_{v,j}\ge 2$ and so we also know $2g_{v,j}-2\ge g_{v,j}$, It follows that the value of the expression in \eqref{eq: minimum} is at least $g_{v,1}$. 

Moreover, recall from Example \ref{QBB} that for any $h\ge 1$ the Satake--Baily--Borel compactification $\mathcal{A}_h[n]^*$ of $\mathcal{A}_h[n]$ has boundary $\mathcal{A}_h[n]^*\backslash\mathcal{A}_h[n]$ of codimension $h$ in $\mathcal{A}_h[n]^*$. Therefore, the boundary of the Satake--Baily--Borel compactification $\mathcal{A}_{g_{v,1}}[n]^*\times \cdots \times \mathcal{A}_{g_{v,s}}[n]^*$ of $\mathcal{A}_{g_{v,1}}[n]\times \cdots \times \mathcal{A}_{g_{v,s}}[n]$ has codimension $g_{v,1}$ in $\mathcal{A}_{g_{v,1}}[n]^*\times \cdots \times \mathcal{A}_{g_{v,s}}[n]^*$. Therefore, letting
\[\mathcal{Y}^*=\{A_{c,1}\}\times \cdots \times \{A_{c,r}\}\times \mathcal{A}_{g_{v,1}}[n]^*\times \cdots \times \mathcal{A}_{g_{v,s}}[n]^*,\]
it follows that $\mathcal{Y}^*$ is a compactification of $\mathcal{Y}$ such that $\mathcal{Y}^*\backslash \mathcal{Y}$ has codimension $g_{v,1}$ in $\mathcal{Y}^*$. 

Thus by Proposition \ref{bb prop}, we have that for any $d\le g_{v,1}-1$, a general complete intersection of ample divisors on $\mathcal{Y}^*$ yields a $d$-dimensional smooth projective subvariety $X$ of $\mathcal{Y}\backslash (\mathcal{Y}\cap \Amdec)$. Moreover since the abelian varieties $A_{v_j}$ are chosen to be general in $\mathcal{A}_{g_{v,j}}[n]$ (after imposing some level $n$ structure) such a general complete intersection can be chosen to have very general fiber 
\[A_{c,1}\times \cdots \times A_{c,r}\times A_{v,1}\times \cdots \times A_{v,s}.\]

Therefore we have produced a $d$-dimensional smooth projective variety $X$ such that a very general fiber $A$ of $\alpha X$ satisfies $A\sim A_{c,1}\times \cdots \times A_{c,r}\times A_{v,1}\times \cdots \times A_{v,s}$ and such that for any general $\alpha \in Sp(2g,\Q)$ we have $\alpha X \cap \Adec=\emptyset$. 
After a possible base change, this $\alpha X$ yields a family $\mathscr{A}\rightarrow B$ of maximal variation of indecomposable $g$-dimensional abelian varieties. It just remains to compute the connected monodromy group $T(\mathscr{A})$ of this family. 
 
Since isogenous abelian varieties have isomorphic rational cohomology,  Proposition \ref{bb prop} implies that  the connected monodromy group $T(\mathscr{A})$ is a normal subgroup of 
\[\mathcal{D}MT\left(\bigoplus_{j=1}^s V_{v,j}\right),\]
where $V_{v,j}=H^1(A_{v,j},\Q)$. 

Since each $A_{v,j}$ is general in $ \mathcal{A}_{g_{v,j}}[n]$ and so in particular $\End_\Q(A_{v,j})=\Q$ for all $j$, we have (see \cite[Theorem 3.2]{MZ}, \cite{hazama})
\[\mathcal{D}MT\left(\bigoplus_{j=1}^s V_{v,j}\right)=\prod_{j=1}^s \mathcal{D}MT(V_{v,j}).\]

Moreover, since each $A_{v,j}$ is general in $ \mathcal{A}_{g_{v,j}}[n]$, it follows that $ \mathcal{D}MT(V_{v,j})=Sp_\Q(V_{v,j})$ for every $1\le j\le s$ (see for instance \cite[Prop. 17.4.2]{BL}). Hence
\[\mathcal{D}MT\left(\bigoplus_{j=1}^s V_{v,j}\right)=\prod_{j=1}^sSp_\Q(V_{v,j}).\]

However by construction of $X$ as a general complete intersection in $\mathcal{Y}^*$, we know that the monodromy representation of $\pi_1(X)$ acts non-trivially on all the $V_{v,j}$ for $1\le j\le s$. Since  each $Sp_\Q(V_{v,j})$ is simple as a $\Q$-algebraic group, it follows that
\[T(\mathscr{A})=\prod_{j=1}^sSp_\Q(V_{v,j}).\]
\end{proof}

As discussed in Section \ref{intro} one of the motivations for the above theorem comes from the question of trying to understand 
 groups arising as connected monodromy groups of fibered projective varieties in the vein of  for instance \cite{arapura}.  The above theorem thus yields the following immediate consequence:

\begin{cor}
For any positive integer $d\le \min_i g_i-1$, where $g_i\ge 2$ for all $i$, and any $g'\ge \sum_{i=1}^sg_i$, the group  
\[T=\prod_{j=1}^sSp(2g_i,\Q)\]
arises as the connected monodromy group of a $d$-dimensional family of maximal variation of indecomposable $g'$-dimensional abelian varieties over a smooth projective base.
\end{cor}


\section{Applications to Kodaira Fibrations}\label{section kodaira}

As also discussed in Section \ref{intro}, another motivation for the above results comes from trying to understand Kodaira fibrations and their possible monodromy groups. 
Recall from Section \ref{intro kod section} that a Kodaira fibration is a non-isotrivial fibration $f\colon S\rightarrow B$ of a smooth projective surface to a smooth algebraic curve $B$ such that all fibers $F$ of $f$ are smooth algebraic curves. Applying Theorem \ref{main theorem} in the case $g=3$ recovers a result originally obtained in \cite{FKod}:

\begin{cor}\label{cor kod reproof}
For any choice of elliptic curve $E$, there exist infinitely many Kodaira fibrations $f\colon S\rightarrow B$ with fiber genus $3$ such that the cohomology of a very general fiber $F$ decomposes as $H^1(F,\Q)=H^1(E,\Q)\oplus W$ and the connected monodromy group of $S$ is $Sp_\Q(W)\cong Sp(4,\Q)$. 
\end{cor}

\begin{proof}
Let $g=3$, $h_c=1$, and $h_v=2$, with $\lambda_c=(1)$ and $\lambda_v=(12)$. By choosing $A_{c,1}=E$ and letting $d=1$, 
Theorem \ref{main theorem} yields a $1$-dimensional family of indecomposable abelian threefolds over a smooth algebraic curve $C$ with the given monodromy group and whose fibers have the given cohomology decomposition. To show that this curve $C$ actually parametrizes a family of curves, note first of all that all indecomposable abelian threefolds lie in the image of Torelli map $\mathcal{M}_3[n]\rightarrow \mathcal{A}_3[n]$. The slight subtlety arises from the fact that this map is ramified exactly at the hyperelliptic locus, which has codimension $1$ in $\mathcal{M}_3[n]$ and is affine. Since our complete curve $C$ is obtained as a general complete intersection, although  $C$ will not avoid the hyperelliptic locus, we can ensure that $C$ intersects the hyperelliptic locus at finitely many general points $P_1,\ldots, P_k$. Hence we do have a family of curves over $C-\{P_1,\ldots, P_k\}$. As described for instance in \cite[Section 1.2.1]{catanese}, taking a double cover $\pi\colon C'\rightarrow C$ branched over a set of points that includes $P_1,\ldots, P_k$ yields a family of genus $3$ curves with the desired property over the whole base $C'$, since the Kuranishi family of the each of the curves determined by the $P_i$ has a map to $\mathcal{M}_3[n]$ which is a double covering ramified over the hyperelliptic locus. 
\end{proof}

Additionally, although it doesn't follow directly, we may use the ideas of the proof of Theorem \ref{main theorem} 
to produce the following new construction of Kodaira fibrations with known connected monodromy group.
\begin{cor}\label{kod cor}
For any choice of elliptic curve $E$, there exist infinitely many Kodaira fibrations $f\colon S\rightarrow B$ with fiber genus $4$ such that the cohomology of a very general fiber $F$ decomposes as $H^1(F,\Q)=H^1(E,\Q)\oplus W$ and the connected monodromy group of $S$ is $Sp_\Q(W)\cong Sp(6,\Q)$. 
\end{cor}

\begin{proof}
Let $g=4$, $h_c=1$, and $h_v=3$, with $\lambda_c=(1)$ and $\lambda_v=(123)$. Consider, as in the proof of Theorem \ref{main theorem}, the subvarieties of $\mathcal{A}_4[n]$
\[\mathcal{Z}=\mathcal{A}_{1}[n]\times \mathcal{A}_{3}[n]\]
\[\mathcal{Y}=\{E\}\times \mathcal{A}_{3}[n].\]
We know from Lemma \ref{lem dim2}, that 
the subvariety $\mathcal{Y}\cap \Amdec$ of $\mathcal{Y}$ has codimension $3$ in $\mathcal{Y}$. We know by Corollary \ref{main cor} that for $\alpha \in Sp(8,\Q)$ general, 
we have $\alpha \mathcal{Z}\cap \Adec \subset \Amdec$ and hence $\alpha \mathcal{Y}\cap  \Adec \subset \Amdec$. Thus $\alpha \mathcal{Y}\cap  \Adec$ has codimension at least $3$ in $\alpha \mathcal{Y}$. 

Additionally, we know that the Torelli locus $\mathcal{T}_4$ of $\mathcal{A}_4[n]$, meaning the locus of Jacobians of curves, has codimension $1$ in $\mathcal{A}_4[n]$. Hence the subvariety  $\mathcal{W}=\alpha \mathcal{Y}\cap \mathcal{T}_4$ has codimension at most $1$ in $\alpha\mathcal{Y}$. Moreoever since  $\alpha \mathcal{Y}\cap  \Adec$ has codimension at least $3$ in $\alpha \mathcal{Y}$, it follows that $\mathcal{W}\cap \Adec$ has codimension at least $2$ in $\mathcal{W}$. 

Now consider the compactification $\mathcal{Y}^*=\{E\}\times  \mathcal{A}_{3}[n]^*$  obtained by taking the Satake--Baily--Borel compactification of $\mathcal{A}_3[n]$. Since the boundary  $\mathcal{A}_{3}[n]^*\backslash \mathcal{A}_3[n]$ has codimension $3$ in $\mathcal{A}_{3}[n]^*$ (see Example \ref{QBB}), it follows that the boundary $\mathcal{Y}^*\backslash \mathcal{Y}$ has codimension $3$ in $\mathcal{Y}^*$. Hence letting $\mathcal{W}^*$  be the compactification of $\mathcal{W}$ induced by the compactification $\mathcal{Y}^*$, we have that the boundary $\mathcal{W}^*\backslash \mathcal{W}$ has codimension at least $2$ in 
 $\mathcal{W}^*$. 
 
 Applying Proposition \ref{bb prop} to $\mathcal{W}$ and its subvariety $\mathcal{W}\cap \Adec$  we then obtain a smooth algebraic curve $C$ obtained as a general complete intersection such that $C$ parametrizes a family of indecomposable abelian fourfolds whose very general fiber  has cohomology with the given decomposition. Moreover, since for a general abelian threefold $A\in \mathcal{A}_3[n]$ with $W=H^1(A,\Q)$, we have $\mathcal{D}MT(W)=Sp_\Q(W)$ \cite[Prop 17.4.2]{BL}, which is simple as a $\Q$-algebraic group, Proposition \ref{bb prop} also implies that this family has connected monodromy group $Sp_\Q(W)\cong Sp(6,\Q)$.
 
 As in the proof of Corollary \ref{cor kod reproof}, it remains to argue that this curve $C$ actually parametrizes a complete family of curves. Identically to the proof in that case, 
 since the Torelli map $\mathcal{M}_4[n]\rightarrow \mathcal{A}_4[n]$ is an immersion away from the hyperelliptic locus $\mathcal{H}_4[n]$ which is affine and of codimension $2$ in $\mathcal{M}_4[n]$, we may ensure that if this complete curve $C$, obtained as a general complete intersection, intersects $\mathcal{H}_4[n]$, it intersects it in finitely many general points. Then after taking a double cover $C'\rightarrow C$ branched over a set containing these points, we can ensure that $C'$ really does parametrize a family of curves with the desired properties. 
 
 \end{proof}


\section{Other connected monodromy groups}\label{section unitary}
We remark here that the methods of Theorem \ref{main theorem} enable the construction of complete families of indecomposable abelian varieties with other kinds of connected monodromy groups. 
If $\mathcal{V}$ is the subvariety of $\mathcal{A}_g[n]$ parametrizing abelian varieties with the desired decomposition and Hodge-theoretic data, then $\mathcal{V}$ has a Baily-Borel compactification $\mathcal{V}^*$. Hence as long as the codimensions of both $\mathcal{V}\cap \Amdec$ in $\mathcal{V}$ and the boundary $\mathcal{V}^*\backslash \mathcal{V}$ in $\mathcal{V}^*$ are at least $2$, one can produce via Proposition \ref{bb prop} complete families of the desired type as complete intersections on $\mathcal{V}$. 

As an illustration of the above in practice, we detail below the construction of a complete family of indecomposable abelian varieties where the connected monodromy group  is a unitary group due to the fact that the general fiber of the family has endomorphisms by a fixed imaginary quadratic field.

\subsection{Imaginary quadratic multiplication and the Mumford-Tate group}\label{section MTunitary}
For a fixed imaginary quadratic field $L$ and positive integers $p,q$, recall the notation $\mathcal{Z}_L(p,q)$ introduced in Example \ref{unitaryBB} to denoting the $pq$-dimensional Shimura variety in $\mathcal{A}_{p+q}[n]$ parametrizing $(p+q)$-dimensional abelian varieties $A$ with  level $n$ structure such that $L\hookrightarrow \End_\Q(A)$ via the prescribed action given by the integers $(p,q)$. 

Let $V=H^1(A,\Q)$ for $A$ a general point of $\mathcal{Z}_L(p,q)$. Let $D_V\cong M_{p+q}(L^{\mathrm{op}})$ denote the centralizer of $L$ in $\End_\Q(V)$, let $^{-}$ denote the involution on $D_V$ induced by complex conjugation on $L$.  Then define
\begin{align*}
U(D_V,^{-})&\coloneqq\{d\in D_V^*\mid \overline{d}=d^{-1}\}\\
SU(D_V,^{-})&\coloneqq \ker \left(\mathrm{Norm}_{D_V/\Q}\colon U(D_V,^{-})\rightarrow \mathbb{G}_{m,\Q}\right).
\end{align*}
The groups $U(D_V,^{-})$ and $SU(D_V,^{-})$ are $\Q$-forms of the groups $U(p,q)$ and $SU(p,q)$ respectively (see for instance \cite{boi}).

\begin{lem}\label{lem mtunitary}
For a fixed imaginary quadratic field $L$ and positive integers $p,q$, let $A$ be a general point of the Shimura variety $\mathcal{Z}_L(p,q)$ and $V=H^1(A,\Q)$. Then $\mathcal{D}MT(V)=SU(D_V,^{-})$. 
\end{lem}

\begin{proof}
The argument here is inspired by similar arguments in \cite[pg 4110]{totaro} and \cite[Proposition VI. A.5]{domain}. The endomorphism algebra of any $\Q$-Hodge structure $W$ is always the commutant of the Mumford-Tate group  $MT(W)$ in $\End_\Q(W)$. Since Mumford-Tate groups are connected, it follows that $M\coloneqq MT(V)$ is contained in the connected component $H$ of the centralizer of $L$ in $Sp(V)$. Under the given hypotheses, we have $H=U(D_V,^{-})$. 

Now note that since $M$ is the Mumford-Tate group of a general point of $\mathcal{Z}_L(p,q)$, the group $M$ is determined solely by the data of the polarized $\Q$-Hodge structure $V=H^1(A,\Q)$ together with the embedding $L\hookrightarrow \End_\Q(A)$ via the choice of $(p,q)$ and the connected component $\mathcal{Z}_L(p,q)$. However by construction $H(\Q)$ preserves $V=H^1(A,\Q)$ together with its action by $L$ and so, because $H$ is connected, the group $H(\Q)$ also preserves $\mathcal{Z}_L(p,q)$. 
Hence, since $H$ is a connected group over the perfect field $\Q$, we have that  $H(\Q)$ normalizes the group $M$ \cite[Corollary 18.3]{borel}. So, since $M\subset H$, it follows that $M$ is in fact a normal subgroup of $H=U(D_V,^{-})$. 

Since the group $SU(D_V,^{-})\subset U(D_V,^{-})$ is $\Q$-simple, we thus have an inclusion of normal subgroups
\[SU(D_V,^{-})\triangleleft M \triangleleft U(D_V,^{-}).\]
However the derived subgroup $\mathcal{D}U(D_V,^{-})$ of the group $U(D_V,^{-})$ is exactly the group $SU(D_V,^{-})$. Therefore we must have $\mathcal{D}M=SU(D_V,^{-})$.
\end{proof}


\begin{thm}\label{thm unitary}
Fix pairwise non-isogenous elliptic curves $E_1, \ldots, E_r$ without CM as well as an imaginary quadratic field $L$, a choice of positive integers $p,q$ such that $p+q\ge4$, and an abelian variety $Z$ which is general in the moduli space $\mathcal{Z}_L(p,q)[n]$. 
Then for any positive integer
\[d\le \min(2p, p+q-2, 2q)-1\]
 there exist infinitely many complete families  $f\colon \mathscr{A}\rightarrow S$ of maximal variation of abelian varieties of dimension $g=p+q+r$ with the properties:
\begin{enumerate}
\item The base $S$ is a smooth projective variety of dimension $d$
\item Every fiber of $f$ is indecomposable.
\item If $A$ is a very general fiber of $f$, then $A$ is isogenous to the product
$E_1\times \cdots \times E_r\times Z,$ where $\End_\Q(Z)=L$ with prescribed action determined by $(p,q)$
\item Letting $V=H^1(Z,\Q)$, the connected monodromy group of $\mathscr{A}$ is $SU(D_V,^{-})$.
\end{enumerate}
\end{thm}

\begin{proof}
Consider the locus in $\mathcal{A}_g[n]$ given by
\[\mathcal{Y}=\{E_1\}\times \cdots \times \{E_r\}\times  \mathcal{Z}_L(p,q).\]
Let $\alpha$ be any general element of $Sp(2g,\Q)$. Then it follows from Corollary \ref{main cor} that $\alpha \mathcal{Y}\cap \Adec\subset \Amdec$. It follows from Lemma \ref{lem dim4} that  $\mathcal{Y}\cap \Amdec$ has codimension $\min(2p, p+q-2, 2q)$ in $\mathcal{Y}$. 

Additionally, as discussed in Example \ref{unitaryBB}, the Satake--Baily--Borel compactification 
$\mathcal{Z}_L(p,q)^*$ of $\mathcal{Z}_L(p,q)$ has boundary of codimension at least $p+q-1$. 
Thus
\[\mathcal{Y}^*=\{E_1\}\times \cdots \times \{E_r\}\times  \mathcal{Z}_L(p,q)^*\] is a compactification of $\mathcal{Y}$ such that $\mathcal{Y}^*\backslash \mathcal{Y}$ has codimension at least at least $p+q-1$ in $\mathcal{Y}^*$. Hence by Proposition \ref{bb prop}, it follows that for any $d\le \min(2p, p+q-2, 2q)-1$ a general complete intersection of ample divisors on $\mathcal{Y}^*$ yields a $d$-dimensional smooth compact subvariety $X$ of $\mathcal{Y}\backslash( \mathcal{Y}\cap \Amdec)$. Moreover we can ensure that this general complete intersection $X$ has the property that a very general fiber $A$ of the universal family $\mathscr{U}$ of $X$ will be of the form $A\sim E_1\times \cdots \times E_r\times Z,$ for  $Z$ the given general point of $\mathcal{Z}_L(p,q)$. 

By  Proposition \ref{bb prop},   the connected monodromy group $T(\mathscr{U})$  is then a normal subgroup of $\mathcal{D}MT(V)$, where $V=H^1(Z,\Q)$. Since by Lemma \ref{lem mtunitary}, we have $\mathcal{D}MT(V)=SU(D_V,^{-})$ and $SU(D_V,^{-})$ is a $\Q$-simple group, it follows that $T(\mathscr{U})=SU(D_V,^{-})$. 

Therefore for any general element $\alpha$ of $Sp(2g,\Q)$, we have that $\alpha X$ is a $d$-dimensional smooth algebraic subvariety of $\alpha \mathcal{Y}$ such that $\alpha X \cap  \Adec =\emptyset$. So indeed after possible base change, this $\alpha X$ yields a complete family of indecomposable $g$-dimensional abelian varieties with the desired properties. 
\end{proof}

Once again, in the direction of understanding 
 groups arising as connected monodromy groups of fibered projective varieties, the above theorem thus yields the following immediate consequence:

\begin{cor}\label{unitary cor}
Let $L$ be a fixed imaginary quadratic field. For any positive integers $p,q,$ and $g'$ such that $5\le p+q+1\le g'$ and any positive integer $d\le \min(2p, p+q-2, 2q)-1,$
the group $SU(D_V,^{-})$, where $V\cong H^1(A,\Q)$ for some simple abelian variety $A\in \mathcal{Z}_L(p,q)$, arises as the connected monodromy group of a $d$-dimensional complete family of maximal variation of $g'$-dimensional indecomposable abelian varieties over a smooth algebraic variety.
\end{cor}

\begin{proof}
Let $m=g'-(p+q)$. Note that by assumption $m\ge 1$. Hence let fix pairwise non-isogenous elliptic curves  $E_1, \ldots, E_m$ without CM. The result then follows from Theorem \ref{thm unitary}.
\end{proof}

\textbf{Acknowledgements.}The author would like to thank Giulia Sacc\`a and Nick Salter whose conversations inspired the writing down of this paper. 

The author
 gratefully acknowledges support of the National Science Foundation through award DMS-1803082.  Additionally, this material is based upon work supported by the National Science Foundation under Grant DMS-1440140 while the author was in residence at the Mathematical Sciences Research Institute in Berkeley, California during the Spring 2019 semester.

\normalsize{\bibliography{Kodaira.bib}}
\bibliographystyle{alpha}

\end{document}